\documentclass[11pt,oneside,reqno]{amsart}
\usepackage{amsmath,amsthm,amsfonts,amssymb,bbm}
\usepackage[usenames,dvipsnames]{color}
\usepackage{graphicx,psfrag,subfigure,url}
\usepackage{enumitem}
\usepackage[noadjust]{cite}
\usepackage{mathrsfs,mathtools}
\usepackage[extension=pdf]{hyperref}
\usepackage{comment}
\usepackage[totalheight=9.2in,totalwidth=5.95in,centering]{geometry}
\usepackage{epstopdf}
\definecolor{darkblue}{rgb}{0.13,0.13,0.39}
\hypersetup{colorlinks=true,urlcolor=darkblue,citecolor=darkblue,linkcolor=darkblue,pdftitle={Determinantal
  line ensembles}}

\usepackage{setspace}

\newcommand{\EE}{\ensuremath{\mathbb{E}}}

\newcommand{\PP}{\ensuremath{\mathbb{P}}}
\newcommand{\R}{\ensuremath{\mathbb{R}}}

\newcommand{\Z}{\ensuremath{\mathbb{Z}}}

\newcommand{\Y}{\ensuremath{\mathbb{Y}}}

\newcommand{\qand}{\quad\text{and}\quad}
\newcommand{\qqand}{\qquad\text{and}\qquad}

\newcommand{\I}{{\rm i}}
\newcommand{\pp}{\mathbb{P}}

\newcommand{\ee}{\mathbb{E}}
\newcommand{\rr}{\mathbb{R}}

\newcommand{\zz}{\mathbb{Z}}
\newcommand{\aip}{\mathrm{Airy}_2}
\newcommand{\aipo}{\mathrm{Airy}_1}
\newcommand{\Bt}{\mathrm{Airy}_{2\to 1}}
\newcommand{\prc}{\mathcal{P}}
\newcommand{\ch}{\mathcal{H}}

\newcommand{\cb}{\mathcal{B}}

\newcommand{\cm}{\mathcal{M}}
\newcommand{\cw}{\mathcal{W}}

\newcommand{\oQ}{\overline{Q}}

\newcommand{\uno}[1]{\mathbf{1}_{#1}}

\newcommand{\wt}{\widetilde}

\newcommand{\KAirytwo}{K_{2}}
\newcommand{\KAirytwoext}{K^{{\rm ext}}_{2}}

\newcommand{\KPrc}{K_{{\rm Prc}}}
\newcommand{\KPrcext}{K^{{\rm ext}}_{{\rm Prc}}}

\newcommand{\Zproc}{\mathcal{Z}}

\newcommand{\KZ}{K_{{z,z',\xi}}}
\newcommand{\KZext}{K^{{\rm ext}}_{{z,z',\xi}}}

\newcommand{\DZ}{D_{{z,z',\xi}}}

\newcommand{\Ktwoone}{K_{2\to 1}}

\newcommand{\Ktwooneext}{K^{{\rm ext}}_{2\to 1}}

\newcommand{\KHrm}{K_{{\rm GUE},N}}
\newcommand{\KHrmext}{K^{{\rm ext}}_{{\rm GUE},N}}
\newcommand{\widetildeKHrm}{\widetilde{K}_{{\rm GUE},N}}

\newcommand{\sW}{\mathsf{W}}
\newcommand{\sK}{\mathsf{K}}
\newcommand{\sU}{\mathsf{U}}
\newcommand{\sV}{\mathsf{V}}
\newcommand{\tsWo}{\mathsf{W_3}}
\newcommand{\tsWt}{\mathsf{W_4}}
\newcommand{\sWo}{\mathsf{W_1}}
\newcommand{\sWt}{\mathsf{W_2}}
\newcommand{\sQ}{\mathsf{Q}}

\newcommand{\sI}{\mathsf{I}}

\DeclareMathOperator{\Ai}{Ai}
\DeclareMathOperator{\tr}{tr}

\newcommand{\hs}{\hspace{0.05em}}

\newtheorem{theorem}{Theorem}[section]

\newtheorem{lemma}[theorem]{Lemma}
\newtheorem{proposition}[theorem]{Proposition}
\newtheorem{corollary}[theorem]{Corollary}

\theoremstyle{definition}
\newtheorem{remark}[theorem]{Remark}

\theoremstyle{definition}

\theoremstyle{definition}
\newtheorem{definition}[theorem]{Definition}

\theoremstyle{definition}

\newtheorem{assumption}{Assumption}

\setlist{itemsep=3pt,topsep=2pt,leftmargin=6ex,label={\footnotesize\textbullet}}

\title{Multiplicative functionals on ensembles of non-intersecting paths}

\author[A. Borodin]{Alexei Borodin}
\address{A. Borodin,
Massachusetts Institute of Technology,
Department of Mathematics,
77 Massachusetts Avenue, Cambridge, MA 02139-4307, USA
and
Institute for Information Transmission Problems,
Moscow, Russia}
\email{borodin@math.mit.edu}

\author[I. Corwin]{Ivan Corwin}
\address{I. Corwin,
Massachusetts Institute of Technology,
Department of Mathematics,
77 Massachusetts Avenue, Cambridge, MA 02139-4307, USA}
\email{ivan.corwin@gmail.com}

\author[D. Remenik]{Daniel Remenik} \address{D. Remenik, Departamento de Ingenier\'{i}a
  Matem\'{a}tica and Centro de Modelamiento Matem\'atico, Universidad de Chile, Av. Blanco
  Encalada 2120, Santiago, Chile} \email{dremenik@dim.uchile.cl}

\begin{document}
\sloppy
\maketitle
\thispagestyle{empty}

\begin{abstract}
  The purpose of this article is to develop a theory behind the occurrence of
  ``path-integral'' kernels in the study of extended determinantal point processes and
  non-intersecting line ensembles. Our first result shows how determinants involving such
  kernels arise naturally in studying ratios of partition functions and expectations of multiplicative functionals for ensembles of
  non-intersecting paths on weighted graphs. Our second result shows how Fredholm
  determinants with extended kernels (as arise in the study of extended determinantal
  point processes such as the Airy$_2$ process) are equal to Fredholm determinants with
  path-integral kernels. We also show how the second result applies to a number of
  examples including the stationary (GUE) Dyson Brownian motion, the Airy$_2$ process, the
  Pearcey process, the Airy$_1$ and Airy$_{2\to 1}$ processes, and Markov processes on
  partitions related to the $z$-measures.
\end{abstract}

\section{Introduction}

The Airy$_2$ process is a universal scaling limit of a wide variety of probabilistic
systems including random matrix theory, random growth processes, interacting particle
systems and directed polymers in random media (see \cite{PatrikReview,quastelRemReview}
and references therein).
Denoted $\aip(\cdot)$, it is defined via its consistent finite dimensional distributions: for
$t_1<t_2<\cdots< t_n$,
\begin{equation}\label{e1}
\PP\!\left(\bigcap_{i=1}^{n} \left\{ \aip(t_i) \leq s_i\right\} \right) = \det(I-\chi \KAirytwoext)_{L^2(\{t_1,\ldots, t_n\}\times \R,\mu)}.
\end{equation}
Here $\chi$ is an operator which acts on functions $f\!:\{t_1,\ldots,t_n\}\times \R\to \R$ as
$$
\chi f(t_i,x) := \mathbf{1}_{x\geq s_i} f(t_i,x).
$$
The operator $\KAirytwoext$ acts as
$$
\KAirytwoext f(t_i,x) := \sum_{j=1}^{n} \int_{\R} dy\,\KAirytwoext(t_i,x;t_j,y) f(t_j,y)
$$
where $\KAirytwoext(s,x;t,y)$ is the ``extended'' Airy$_2$ kernel given by
$$
\KAirytwoext(s,x;t,y):=
\begin{cases}
\int_0^{\infty} d\lambda\,e^{-\lambda(s-t)} \Ai(x+\lambda)\Ai(y+\lambda) & \textrm{if } s\geq t,\\
-\int_{-\infty}^{0} d\lambda\,e^{-\lambda(s-t)} \Ai(x+\lambda)\Ai(y+\lambda) & \textrm{if } s<t,
\end{cases}
$$
with $\Ai(x)$ the classical Airy function. The right-hand side of (\ref{e1}) is the
Fredholm determinant of the identity minus a trace class operator (see Section
\ref{Fredsec} for definition and details) and the measure $\mu$ appearing there is the
product of counting measure on $\{t_1,\ldots, t_n\}$ and Lebesgue measure on $\R$.


The formula given in (\ref{e1}) for the finite dimensional distributions of the Airy$_2$
process becomes increasingly cumbersome as $n$ increases. This is due to the
$n$-dependence in the $L^2$ space on which the operators act. When taking a limit of a
sequence of operators, or their determinants, it is convenient to have the operators all
act on the same $L^2$ space, rather than a sequence of different spaces.

In Pr\"{a}hofer and Spohn's initial work on the Airy$_2$ process (see Section 5 of
\cite{PS} for $n=2$ or \cite{prolhacSpohn,CQR,quastelRemAiry1} for $n\geq 2$) the extended kernel formula
is shown to be equivalent to the following ``path-integral'' kernel formula:
\begin{equation}\label{e2}
\PP\!\left(\bigcap_{i=1}^{n} \left\{ \aip(t_i) \leq s_i\right\} \right) = \det(I-\KAirytwo +
\bar P_{s_1} e^{(t_1-t_2)H}\bar P_{s_2} \cdots e^{(t_{n-1}-t_n)H}\bar P_{s_n} e^{(t_n-t_1)H} \KAirytwo)_{L^2(\R)}.
\end{equation}
Here $\KAirytwo(x,y) =\KAirytwoext(0,x;0,y)$ is called the Airy$_2$ kernel, $\bar
P_{s}g(x) = \mathbf{1}_{x\leq s} g(x)$ is a projection operator and $H=-\Delta + x$ is
called the Airy Hamiltonian because $H\!\Ai(\cdot-s) = s\!\Ai(\cdot-s)$ ($\Delta$ is the
Laplacian on $\R$). The dependence on $n$ has been absorbed into the operator rather than
the $L^2$ space, and it is now plausible to take a large $n$ limit.

The reason we call this a path-integral kernel is because a portion of it can be written
in terms of the expectation of a certain path-integral. By the Feynman-Kac
formula (cf. \cite{Karatzas-Shreve}),
\begin{equation}
\bar P_{s_1} e^{(t_1-t_2)H}\bar P_{s_2} \cdots e^{(t_{n-1}-t_n)H}\bar P_{s_n}  f(x) = \EE_{b(t_1)=x}\!
\left[f(b(t_n)) e^{-\int_{t_1}^{t_n} b(s)\hs ds} \prod_{i=1}^{n} \mathbf{1}_{b(t_i)\leq s_{i}}\right],\label{eq:FKPs}
\end{equation}
where $b\!:[t_1,t_n]\to \R$ is the trajectory of a Brownian motion with diffusion
coefficient 2 starting at $b(t_1)=x$.

Let $t_1<\cdots<t_n$ fill out the interval $[\ell,r]$ and let $s_i=h(t_i)$ for some function
$h\!:[\ell,r]\to~\R$. Then as $n$ goes to infinity, the above operator has a limit in trace
norm (see \cite{CQR} or Section \ref{Airy2Sec} below for details) given by
$$
\Gamma^{h}_{\ell,r}f(x) = \EE_{b(\ell)=x}\!\left[ f(b(r)) e^{-\int_{\ell}^{r}b(s)\hs ds}
  \mathbf{1}_{b\leq h}\right],
$$
where $\{b\leq h\}$ denotes the event $\{b(s)\leq h(s),\, \forall s\in [\ell,r]\}$. Thus,
it is shown in \cite{CQR} Theorem 2 or equation (\ref{continuumstatAiry}) below that
\begin{equation}
\PP\!\left(\aip(s) \leq h(s),\, \forall s\in [\ell,r]\right) = \det(I-\KAirytwo + \Gamma^h_{\ell,r} e^{(r-\ell)H} \KAirytwo)_{L^2(\R)}.\label{eq:cqr}
\end{equation}
The above formula proved useful in \cite{CQR} in providing a direct proof that the value
of the maximum of the Airy$_2$ process minus a parabola is distributed according to the
(GOE) Tracy-Widom distribution; and in \cite{MQR} in computing the joint distribution for
the value and location (in $t$) of the maximum. The two-time path-integral kernel formula
in \cite{PS} was utilized to compute asymptotics of the two-time covariance of the
Airy$_2$ process, since the extended kernel does not easily yield this. Note that the
left-hand side in the last formula presupposes the existence of a continuous version of
the Airy$_2$ process. This was first shown to exist in \cite{PNGJ}.

What is remarkable about formula \eqref{eq:cqr} is that the right-hand side is simple
(despite the cumbersome finite dimensional distributions given above) and the event in
question in the left-hand side has a clear translation into the operator
$\Gamma^h_{\ell,r}$. As a further application of the Feynman-Kac formula as well as the Cameron-Martin-Girsanov formula
(see \cite{CQR}), the integral kernel of $\Gamma^h_{\ell,r}$ can be expressed as
\[\Gamma^h_{\ell,r}(x,y)=\PP_{b(\ell)=x-\ell^2,b(r)=y-r^2}\!\left(b(s)\leq h(s)-s^2\text{ for
    all }s\in[\ell,r]\right),\] where $b$ is now a Brownian bridge run from
$x-\ell^2$ at time $\ell$ to $y-r^2$ at time $r$ (this means that $\Gamma^h_{\ell,r}f(x)=\int
dy\,\Gamma^h_{\ell,r}(x,y)f(y)$).  In other words, the probability that the Airy$_2$
process hits a function $h$ can be expressed as the Fredholm determinant of an operator
which is partly expressed by the probability that a Brownian bridge hits the same
function (minus a parabola).

We will now see how these formulas for the Airy$_2$ process are part of a more general
result.

\subsection{Extended kernels and the path-integral kernel in a general setting}\label{extendedsec}

There are many other examples of extended determinantal point processes (some given in
Section \ref{sec:examples}) and our aim is to find path-integral kernel formulas for these
other processes. This may have further applications, although we do not address them
here. For example, besides the previous work of \cite{prolhacSpohn}, in
\cite{quastelRemAiry1} a path-integral kernel formula was discovered for the Airy$_1$
process and used to prove existence of a continuous version of the process and its
H\"{o}lder regularity.

When the Airy$_2$ process was introduced, it arose as the top layer of the multi-layer
Airy$_2$ process, which will be denoted $\left\{\aip(i;t):i\in \Z_{\geq 1}, t\in \R\right\}$ and is such that $\aip(i;t)>\aip(j;t)$ for $i<j$. As $t$ varies, $\sum_{i=1}^{\infty} \delta_{\aip(i;t)}$ forms an $\R$-indexed collection of point processes which has the structure of an (extended) determinantal point process (see \cite{borDet} and references therein) with correlation kernel $\KAirytwoext(s,x;t,y)$.

This has the following consequence. For $t_1<\cdots< t_n$ fix $q_{t_i}\!:\R\to\R$ and let
$\bar{q}_{t_i} = 1-q_{t_i}$. For $g\!:\{t_1,\ldots, t_n\}\to \R$ define $q(g) := \prod_{i=1}^n q_{t_i}(g(t_i))$ and likewise define $\bar{q}(g)$. Then (given some
conditions on $q$ to ensure convergence -- see Section \ref{Airy2Sec} below)
\begin{equation}\label{e3}
\EE\!\left[\prod_{i=1}^{\infty} \bar{q}\big(\aip(i;\cdot)\big)\right] = \det(I-Q \KAirytwoext)_{L^2(\{t_1,\ldots, t_n\}\times \R,\mu)}
\end{equation}
where $Qf(t_i,x) := q_{t_i}(x) f(t_i,x)$. The left-hand side above is referred to here as
the expectation of a ``multiplicative functional'' of the multi-layer process. When
$q_{t_i}(x) = \mathbf{1}_{x> s_i}$ (and hence $\bar{q}_{t_i}(x) = \mathbf{1}_{x\leq
  s_i}$), the above formula reduces to (\ref{e1}) with $Q=\chi$.

Our first result shows that the type of identity between extended and path-integral kernel
Fredholm determinants which one gets by equating the right-hand sides of (\ref{e1}) and
(\ref{e2}) is quite general and dependent on a few structural properties of the
kernels.

In stating our results presently, we leave out a number of technical assumptions
(see Section \ref{equivsec} for these details) and assume that we have the following
collection of operators on functions $f\!:\{t_1,\ldots, t_n\}\to \R$:
\begin{itemize}
\item For each $1\leq i, j\leq n$, $\cw_{t_i,t_j}$ (with the convention $\cw_{t_i,t_i}=I$);
\item For each $1\leq i\leq n$, $K_{t_i}$;
\item A diagonal operator $Q$ such that $Qf(t_i,\cdot) :=Q_{t_i}f(t_i,\cdot)$ where for $1\leq i\leq n$ and $g:\R\to\R$, $Q_{t_i}g(x) := q_{t_i}(x)g(x)$.
\end{itemize}

\begin{theorem}[Theorem \ref{thm:extendedToBVP}, with technical assumptions suppressed]\label{thm1}
Assume that for all $1\leq i\leq j\leq k\leq n$ the following holds:
\begin{itemize}
\item {Right-invertibility}: $\cw_{t_i,t_j}\cw_{t_j,t_i}K_{t_i}=K_{t_i}$;
\item {Semigroup property}: $\cw_{t_i,t_j}\cw_{t_j,t_k}=\cw_{t_i,t_k}$;
\item {Reversibility relation}: $\cw_{t_i,t_j}K_{t_j}=K_{t_i}\cw_{t_i,t_j}$.
\end{itemize}
Then
\begin{equation*}
\det\!\big(I-QK^{\rm ext}\big)_{L^2(\{t_1,\dots,t_n\}\times \R,\mu)}
=\det\!\big(I-K_{t_1}+\oQ_{t_1}\cw_{t_1,t_2}\oQ_{t_2}\dotsm\cw_{t_{n-1},t_n}\oQ_{t_n}
    \cw_{t_n,t_1}K_{t_1}\big)_{L^2(\R)},
\end{equation*}
where
\begin{equation*}
K^{\rm ext}(t_i,x;t_j,y)=
  \begin{dcases*}
    \cw_{t_i,t_j}K_{t_j}(x,y) & if $i\geq j$,\\
    -\cw_{t_i,t_j}(I-K_{t_j})(x,y) & if $i<j$,
  \end{dcases*}
\end{equation*}
and $\oQ_{t_i}=I-Q_{t_i}$.
\end{theorem}
This is proved in Section \ref{Airy2Sec} essentially via linear algebra.

Letting $\cw_{t_i,t_j} = e^{-(t_j-t_i)H}$, $K_{t_i} = \KAirytwo$ and $Q_{s_i}=P_{s_i}$ we
recover the equality \eqref{e2} for the Airy$_2$ process. More generally, the result
implies that the expectation of a multiplicative functional of the multi-layer Airy$_2$
process \eqref{e3} can be expressed in a similar way, by replacing each $P_{s_i}$ by
$Q_{s_i}$ on the right-hand side of (\ref{e2}).

In Section \ref{sec:examples} we apply this theorem to a variety of examples of extended
determinantal point processes such as the stationary (GUE) Dyson Brownian motion, the
Airy$_2$ process, the Pearcey process, and Markov processes on partitions related to the
$z$-measures.  We also show how the identity applies to signed extended determinantal
point processes such as the Airy$_1$ and Airy$_{2\to 1}$ processes. In the case of the
stationary (GUE) Dyson Brownian motion and the Airy$_2$ process, we also obtain the
continuum limits of the corresponding path-integral kernel formulas (which is likely
doable in other cases as well).

\subsection{Ensembles of non-intersecting paths and the path-integral kernel}

The multi-layer Airy$_2$ process arises as the scaling limit of a variety of ensembles of
non-intersecting paths (see for instance \cite{KJPaths}). The occurrence of extended
kernel determinants in such ensembles is a consequence of the Eynard-Mehta theorem, which
implies the existence of an extended determinantal point process structure \cite{Eyn,Nag,PNGJ,tracyWidom-Diff,BR}. The equivalence of the extended kernel
determinant formula with the path-integral kernel determinant formula which is given in
Theorem \ref{thm1} (see also Theorem \ref{thm:extendedToBVP}) is via linear algebra, but
does not indicate why such a path-integral kernel formula exists. Theorem
\ref{intropaththm} below provides a direct link between ensembles of non-intersecting
paths and path-integral kernel determinant formulas, and its proof boils down to the
Lindstr\"{o}m-Gessel-Viennot Lemma (recorded below as Lemma \ref{LGV}). By first proving
the path-integral kernel determinant formula and then relating it to an extended kernel
determinant formula, this provides another proof of the extended determinant point process
structure for these ensembles (i.e., the Eynard-Mehta theorem), see Section \ref{sec:detstruc}.

Let us now introduce the ensembles of non-intersecting paths. Fix $T\in \Z_{\geq 0}$ and
$N\in \Z_{\geq 1}$. Let $G=(V,E)$ be a finite directed acyclic planar graph with vertex
set $V=V_0\sqcup V_1\sqcup\cdots \sqcup V_T$ (here $\sqcup$ represents the disjoint union
of sets) and directed edge set $E=E_{0\to 1}\sqcup E_{1\to 2} \sqcup \cdots \sqcup
E_{T-1\to T}$ where $E_{n\to n+1}$ only contains edges from $x\to y$ with $x\in V_n$ and
$y\in V_{n+1}$. Here $x\to y$ denotes an edge directed from $x$ to $y$.

As an example, let $V_n$ be vertices of $\Z^2$ of the form $(n,i)$ for $n-i \equiv 0 \mod
2$ (and $|i|\leq M$ for some $M$) and let $E_n$ contain all directed edges from $(x,n)$ to
$(x\pm 1,n+1)$. Paths in this directed graph are trajectories of simple symmetric random
walks (constrained to stay within distance $M$ from the origin), cf. Figure \ref{setup}.

We define a \emph{path} $\pi$ as a sequence of edges $\big(e_0=x_0\to x_1, e_1=x_1\to
x_2,\ldots, e_{T-1}=x_{T-1}\to x_T\big)$ where $x_i\in V_i$ for $0\leq i\leq T$. For such
a path, let $\pi(n)= x_n$ denote the $n^{\rm th}$ vertex in the path.

To edges $e\in E$ we associate weights $w_e\in \R$, and to a path $\pi$ we associate a
weight $w(\pi)$ given by the product of the weights $w(e_n)$ along the edges $e_n$ of
$\pi$. For $x\in V_0$ and $y\in V_T$ we define a transition matrix \[\mathcal{W}(x,y) :=
\sum_{\pi:x\to y} w(\pi)\] where the summation is over all paths $\pi$ from $x$ to $y$.
Instead of considering just a single path, we may consider ensembles of $N$
non-intersecting paths from elements of $V_0$ to elements of $V_T$ (by which we mean paths
which use disjoint collections of vertices). We define the collection of all such paths as
$$\mathcal{N\!.I.}(N) := \Big\{ \Pi = \{\pi_1,\ldots, \pi_N\}\!: \forall\, i,~ \pi_i \textrm{ goes from } V_0 \textrm{ to } V_T \textrm{ and no two paths intersect}\Big\}.$$

We will describe a measure on such an ensemble. This requires the introduction of two
additional families of functions. For $N$ fixed, consider functions $\psi_i\!:V_0\to \R$,
$1\leq i \leq N$, and $\varphi_j\!:V_T \to \R$, $1\leq j\leq N$. Define the weight of $\Pi\in
\mathcal{N\!.I.}(N)$ as
\begin{equation}
W\!t(\Pi) := \det[\psi_{i}(\pi_j(0))]_{i,j=1}^{N}  \left(\prod_{i=1}^{N} w(\pi_i)\right) \det[\varphi_{i}(\pi_j(T))]_{i,j=1}^{N},\label{eq:wtpi}
\end{equation}
and the partition function as
$$Z = \sum_{\Pi\in \mathcal{N\!.I.}(N)} W\!t(\Pi).$$
If $\psi_i$ and $\varphi_j$ are $\delta$-functions then the ends of the paths are fixed.

Assuming that $Z\neq 0$ we may define a measure (not necessarily positive but with total integral 1) on $\Pi\in \mathcal{N\!.I.}(N)$ as
$$\nu(\Pi) =\frac{W\!t(\Pi)}{Z}.$$
When each of the three factors on the right-hand side of \eqref{eq:wtpi} are positive (and
thus, in particular, $\nu$ is a probability measure), one may think of $\nu$ as
follows. The determinants $\det[\psi_{i}(\pi_j(0))]_{i,j=1}^{N}$ and
$\det[\phi_{i}(\pi_j(0))]_{i,j=1}^{N}$ define measures on the collections of $N$ initial
points in $V_0$ and $N$ final points in $V_T$. The weights $w(\pi_i)$ in the middle factor
in \eqref{eq:wtpi} describe the transition probabilities for $N$ independent paths
$(\pi_1,\dotsc,\pi_N)$ connecting these points. The measure $\nu$ is restricted to
non-intersecting paths, and the division by the normalizing constant $Z$ means that $\nu$
corresponds to a measure conditioned on the $N$ paths not intersecting.

Define $\varphi_j^{(0)}\!:V_0\to \R$ by $\varphi_j^{(0)}(x) := \sum_{y\in V_T} \mathcal{W}(x,y)\varphi_j(y)$. Note that this implies that $\mathcal{W}$ has a right-inverse on $\textrm{span}\{\varphi^{(0)}_i\}_{i=1}^N$ which is given by $\mathcal{W}^{-1} \varphi^{(0)}_{j} = \varphi_j$.

We will make the following {\it biorthogonality assumption} on the $\{\psi_i\}_{i=1}^{N}$
and $\{\varphi_j^{(0)}\}_{j=1}^{N}$:
$$ \sum_{x\in V_0} \psi_i(x) \varphi_j^{(0)}(x) = \uno{i=j}.$$

\begin{remark}
  Assuming $Z\neq 0$, one can show that it is always possible to perform a linear
  transformation in $\textrm{span}\{\psi_i\}_{i=1}^{N}$ in such a way that the
  biorthogonality assumption is satisfied and $\nu$ remains unchanged.
\end{remark}

The final concept we introduce is that of a {\it path-integral functional}, which is any function
$$f\!: E_{0\to 1}\times E_{1\to 2} \times \cdots \times E_{T-1\to T}\to \R$$
such that
$$f(e_0,e_1,\ldots, e_{T-1}) = \prod_{n=0}^{T-1} f_n(e_n)$$
for functions $f_n\!:E_{n\to n+1}\to \R$. This definition extends to a directed path $\pi$
from $V_0$ to $V_T$ by setting $f(\pi)$ equal to $f$ applied to the ordered sequence of
edges in $\pi$. From the function $f$ define a second set of edge weights $\tilde{w}_e:=
f_n(e) w_e$ were $n$ is such that $e\in E_{n\to n+1}$. With respect to these weights
$\{\tilde{w}_e\}_{e\in E}$ define a transition matrix \[\widetilde{\mathcal{W}}(x,y) =
\sum_{\pi:x\to y} \tilde{w}(\pi).\]

The following theorem is a consequence of Theorem \ref{combthm}, which is a similar result
for a more general graph setting. The theorem shows how the path-integral kernel
determinant naturally arises from ensembles of non-intersecting paths. A proof of the
below result appears in Section \ref{sec:non-inter}.

\begin{theorem}\label{intropaththm}
For any path-integral functional $f$ as above
\begin{equation*}
\sum_{\Pi=(\pi_1,\ldots, \pi_N)\in \mathcal{N\!.I.}(N)} \prod_{i=1}^{N} f(\pi_i)\nu(\Pi)  = \det(I - K + \widetilde{\mathcal{W}} \mathcal{W}^{-1} K)_{L^2(V_0)}
\end{equation*}
where $K\!:L^2(V_0)\to L^2(V_0)$ is given by its kernel
$$
K(x_1,x_2) = \sum_{i=1}^{N} \varphi^{(0)}_i(x_1) \psi_i(x_2).
$$
\end{theorem}

As we will explain in the proof of Corollary \ref{cor1} below, the above result can also be seen
as a consequence of Theorem \ref{thm1} and the known determinantal structure for ensembles of non-intersecting paths. Instead, we provide a direct and simple linear algebraic proof of Theorem \ref{intropaththm} using only the Lindstr\"{o}m-Gessel-Viennot Lemma. This provides an explanation of the appearance of path-integral kernel formulas.

\subsubsection{Recovering the determinantal structure}\label{sec:detstruc}

As an application of Theorems \ref{thm1} and \ref{intropaththm}, let us see how to recover
the determinantal structure of the ensemble of non-intersecting paths
distributed according to the law $\nu$. We would like to show that for any collection of
vertices $\{x_1,\ldots, x_k\}\in V$, the $\nu$-measure of the set
$$\Big\{\Pi\in \mathcal{N\!.I.}(N): \textrm{ all of the } x_i \textrm{ are visited by paths in } \Pi\Big\}$$
can be written as $\det[K(x_i,x_j)]_{i,j=1}^{k}$ for some fixed matrix $K$ with rows and
columns indexed by the set of vertices $V$. This property can be seen as a consequence of
Corollary \ref{cor1} below which we show.

Consider any collection of functions $q_n\!:V_n\to \R$, $0\leq n \leq T-1$. Consider the
space of matrices with rows and columns indexed by $V$, and for notational convenience
denote $x\in V_n$ as $(n,x)$ so that matrix elements of a matrix $M$ are written as
$M(n,x;m,y)$. Define a matrix $Q$ so that $Qf(n,x) = q_n(x) f(n,x)$. For $m\leq n$ and
$x\in V_m$, $y\in V_n$ define $\mathcal{W}_{m,n}(x,y):= \sum_{\pi:x\to y} w(\pi)$ (for
$m=n$ let this be the identity matrix). For $x\in V_n$ define $\varphi^{(n)}_j(x):=
\sum_{y\in V_T} \mathcal{W}_{n,T}(x,y)\varphi_j(y)$. For $x_1,x_2\in V_n$ define
$K_n(x_1,x_2):= \sum_{i=1}^{N} \varphi^{(n)}_{i}(x_1)\psi_i(x_2)$. Note that for $m\leq
n$, $\mathcal{W}_{m,n}$ has a right-inverse on $\textrm{span}\{\varphi^{(m)}_i\}_{i=1}^N$
which is given by $\mathcal{W}_{m,n}^{-1} \varphi^{(m)}_{j} = \varphi^{(n)}_j$. We will
write this inverse as $\mathcal{W}_{n,m}$. On account of this we may define the following
(extended kernel) matrix
\begin{equation*}
K^{\rm ext}(m,x;n,y)=
  \begin{dcases*}
    \cw_{m,n}K_{n}(x,y) & if $m\geq n$,\\
    -\cw_{m,n}(I-K_{n})(x,y) & if $m<n$.
  \end{dcases*}
\end{equation*}

\begin{corollary}\label{cor1}
For any collection of functions $q_n\!:V_n\to \R$, $0\leq n \leq T-1$
\begin{equation}\label{emeq}
\sum_{\Pi=(\pi_1,\ldots, \pi_N)\in \mathcal{N\!.I.}(N)} \prod_{i=1}^{N} \prod_{n=0}^{T-1} \bar{q}_n(\pi_i(n))\nu(\Pi) = \det(I-QK^{\rm ext})_{L^2(V)}.
\end{equation}
where we recall that $\pi(n)$ denotes the vertex in $V_n$ through which $\pi$ passes and that $\bar{q}(x) = 1-q(x)$.
\end{corollary}

This result is essentially a version of the Eynard-Mehta Theorem.

\begin{proof}
  We use Theorem \ref{thm:extendedToBVP} (the more general version of Theorem \ref{thm1}). The technical
  assumptions are immediately satisfied since we are dealing with a finite vector
  space. The right-invertibility, semigroup property and reversibility relation are all
  readily checked from the definitions of the $\mathcal{W}_{m,n}$ and $K_{n}$. As a
  consequence of that theorem we find that we may rewrite the right-hand side of
  (\ref{emeq}) as
\begin{equation}\label{RHSeq}
  \textrm{RHS}(\ref{emeq}) = \det\!\big(I-K_0 + \oQ_{0}\mathcal{W}_{0,1}
  \oQ_1\mathcal{W}_{1,2}\cdots \oQ_{T-1} \mathcal{W}_{T-1,T} \mathcal{W}_{T,0} K_0\big)_{L^2(V_0)},
\end{equation}
where for $x\in V_n$, $\oQ_n f(n,x) := \bar{q}_n(x) f(n,x)$.

Define a path-integral functional $f$ so that for an edge $e\in E_{n\to n+1}$ from $x\in
E_n$ to $y\in E_{n+1}$, $f_n(e) = \bar{q}_n(x)$. As a consequence, for any path $\pi$ from
$V_0$ to $V_T$, $f(\pi) = \prod_{n=0}^{T-1} \bar{q}_n(\pi(n))$. This observation and
Theorem \ref{intropaththm} then imply that we may rewrite the left-hand side of
(\ref{emeq}) as
\begin{equation}\label{LHSeq}
\textrm{LHS}(\ref{emeq}) = \det\!\big(I - K_0 + \widetilde{\mathcal{W}}_{0,T} \mathcal{W}_{T,0} K_0\big)_{L^2(V_0)}.
\end{equation}
Note that we have introduced the subscripts on the right-hand side to be consistent with
the notation introduced before the statement of the corollary. Due to the specific type of
path-integral functional $f$, it is now straight-forward to see that
$$
\widetilde{\mathcal{W}}_{0,T} = \oQ_{0}\mathcal{W}_{0,1} \oQ_1\mathcal{W}_{1,2}\cdots \oQ_{T-1} \mathcal{W}_{T-1,T}.
$$
This implies that the right-hand sides of equations (\ref{RHSeq}) and (\ref{LHSeq}) match and therefore completes the proof of the corollary.
\end{proof}

\subsection{Outline}
In Section \ref{sec:non-inter} we prove a result about ensembles of non-intersecting paths
on directed graphs (which implies Theorem \ref{intropaththm} above). In Section
\ref{equivsec} we prove a general result (which implies Theorem \ref{thm1} above) showing the equality between certain extended
kernel and path-integral Fredholm determinants. In Section \ref{sec:examples} we apply
this equality between Fredholm determinants to a variety of extended kernels from the
literature, and in Appendix \ref{sec:app} we check the technical assumptions necessary in
order to do this.

\subsection{Acknowledgements}
We thank Senya Shlosman for interest in this project early on. AB was partially supported
by the NSF grant DMS-1056390. IC was partially supported by the NSF grant DMS-1208998 as
well as by the Clay Research Fellowship and by Microsoft Research through the Schramm
Memorial Fellowship. DR was partially supported by the Natural Science and Engineering
Research Council of Canada and by a Fields-Ontario Postdoctoral Fellowship, as well as by
Fondecyt Grant 1120309 and Conicyt Basal-CMM project. DR is appreciative for MIT's
hospitality during the visit in which this project was initiated.

\section{Non-intersecting directed paths on weighted graphs}
\label{sec:non-inter}

\subsection{A general combinatorial result}

Let $G=(V,E)$ be a finite directed acyclic planar graph with vertices $V$ and edges $E$. Fix source vertices $X=\{x_1,\ldots,x_N\}$ and sink vertices $Y=\{y_1,\ldots,y_N\}$. 
Fix edge weights $w_{e}$ for each directed edge $e\in E$.

A directed path $\pi$ is a sequence of vertices connecting a vertex in $X$ to a vertex in $Y$ via directed edges in $E$. Denote the
source vertex of $\pi$ by $\pi(b)$ and the sink vertex of $\pi$ by $\pi(d)$\footnote{We use $b$ to denote the source, or base vertex; and $d$ to denote the sink, or destination vertex.}. We say that two paths {\it intersect} if their vertex sets have non-empty
intersection. To a directed path $\pi$ we associate a weight $w(\pi)$ which is given by
the product of $w_{e}$ over edges $e$ of $\pi$. Define
\begin{equation}
  \mathcal{W}(x,y) = \sum_{\pi:x\to y} w(\pi)\label{eq:defW}
\end{equation}
where the sum is over directed paths $\pi$ from source vertex $x$ to sink vertex $y$.

Define the ensemble of $N$ directed non-intersecting paths from the elements of $X$ to the
elements of $Y$ as
\begin{equation*}
  \mathcal{N\!.I.}(N;X\to Y) = \big\{ \{\pi_1,\ldots, \pi_N\}: \forall i,~  \pi_i(b)\in X,\, \pi_i(d)\in Y \textrm{ and no two paths intersect}\big\}.
\end{equation*}

Write $\Pi=\{\pi_1,\ldots,\pi_N\}$ for an element of $\mathcal{N.I.}(N;X\to Y)$. Note that
the non-intersection condition and the fact that $X$ and $Y$ have $N$ elements ensures
that $\{\pi_1(b),\ldots, \pi_N(b)\} = X$ and $\{\pi_1(d),\ldots, \pi_N(d)\} = Y$.

\begin{lemma}[\mbox{}\cite{KarMc,Ling,GV,Stembridge}]\label{LGV}
  Fix $N\geq 1$. For any finite directed acyclic planar graph $G$ with source vertices
  $X=\{x_1,\ldots,x_N\}$, sink vertices $Y=\{y_1,\ldots, y_N\}$ and edge weights $w_{e}$
  for each directed edge $e\in E$,
\begin{equation*}
\det\!\left[ \mathcal{W}(x_i,y_j)\right]_{i,j=1}^{N} = \sum_{\Pi \in \mathcal{N\!.I.}(N;X\to Y)} \prod_{i=1}^{N} w(\pi_i).
\end{equation*}
\end{lemma}

\begin{figure}
\begin{center}
\includegraphics[width=.8\textwidth]{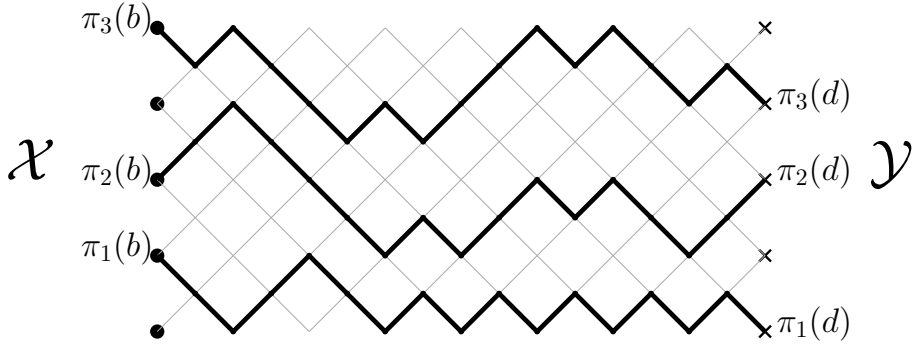}
\end{center}
\caption{An example of a graph $G$ with source vertices $\mathcal{X}$ on the left-hand side and sink vertices $\mathcal{Y}$ on the right-hand side. Here there are $N=3$ non-intersecting paths which are shown in grey, with starting points $\pi_i(b)$ and ending points $\pi_i(d)$ for $i=1,2,3$.}\label{setup}
\end{figure}

Consider now finite sets of source vertices $\mathcal{X}\subset V$ and sink vertices
$\mathcal{Y}\subset V$ (with at least $N$ vertices in each of $\mathcal{X}$ and
$\mathcal{Y}$). For such $\mathcal{X}$ and $\mathcal{Y}$ we can likewise define the
ensemble of $N$ directed non-intersecting paths from elements of $\mathcal{X}$ to elements
of $\mathcal{Y}$. Denote this ensemble $\mathcal{N\!.I.}(N;\mathcal{X}\to\mathcal{Y})$.

Fix functions $\psi_i\!:\mathcal{X}\to \R$ for $1\leq i\leq N$ and functions
$\varphi_j\!:\mathcal{Y}\to \R$ for $1\leq j\leq N$. For $\Pi=\{\pi_1,\ldots, \pi_N\}$
with source vertices $\{\pi_1(b),\ldots,\pi_N(b)\}\subset \mathcal{X}$ and sink vertices
$\{\pi_1(d),\ldots,\pi_N(d)\}\subset \mathcal{Y}$ we define the weight of $\Pi$ as
\begin{equation*}
  W\!t(\Pi) := \det[\psi_{i}(\pi_j(b))]_{i,j=1}^{N}  \left(\prod_{i=1}^{N} w(\pi_i)\right) \det[\varphi_{i}(\pi_j(d))]_{i,j=1}^{N}.
\end{equation*}
Define a partition function for directed non-intersecting ensembles of $N$ paths from
$\mathcal{X}$ to $\mathcal{Y}$ with respect to weights $\{w_e\}_{e\in E}$ and functions
$\{\psi_i\}_{i=1}^{N}$, $\{\varphi_j\}_{j=1}^{N}$ as
\begin{equation*}
  Z = Z\left(\mathcal{X},\mathcal{Y},\{w_e\}_{e\in E}, \{\psi_i\}_{i=1}^{N},\{\varphi_j\}_{j=1}^{N}\right) := \sum_{\Pi\in \mathcal{N\!.I.}(N;\mathcal{X}\to \mathcal{Y})} W\!t(\Pi).
\end{equation*}

For $1\leq j\leq N$ define $\varphi_j^{(b)}\!:\mathcal{X}\to \R$ by
\begin{equation}\label{varphib}
  \varphi_j^{(b)}(x):= \sum_{y\in \mathcal{Y}} \mathcal{W}(x,y) \varphi_j(y)
\end{equation}
and further define the operator $K\!:L^2(\mathcal{X})\to L^2(\mathcal{X})$ by its kernel
(which is just an $\mathcal X$-indexed matrix)
\begin{equation}\label{eq:defK}
  K(x_1,x_2) = \sum_{i=1}^{N} \varphi_i^{(b)}(x_1)\psi_i(x_2).
\end{equation}
Observe that $\mathcal{W}$ has a right-inverse on
$\textrm{span}\{\varphi^{(b)}_i\}_{i=1}^N$, which we will denote by $\mathcal{W}^{-1}$,
given by
\[\mathcal{W}^{-1}\varphi_j^{(b)}=\varphi_j.\]
In particular, since the range of $K$ is contained in
$\textrm{span}\{\varphi^{(b)}_i\}_{i=1}^N$, $\mathcal{W}^{-1}K$ is well defined as an
operator mapping $L^2(\mathcal{X})$ to $L^2(\mathcal{Y})$ :
\begin{equation}
  \label{eq:defWinvK}
  \mathcal{W}^{-1}Kf=\sum_{i=1}^N\left\langle\psi_i,f\right\rangle_{L^2(\mathcal{X})}\varphi_i,
\end{equation}
where $\left\langle\cdot,\cdot\right\rangle_{L^2(\mathcal{X})}$ is the inner product in
$L^2(\mathcal{X})$.

We say that the {\it biorthogonality assumption} is satisfied if
\begin{equation*}
  \left\langle \psi_i,\varphi_j^{(b)}\right\rangle_{L^2(\mathcal{X})} = \uno{i=j} \qquad \textrm{for all } 1\leq i,j\leq N.
\end{equation*}

\begin{theorem}\label{combthm}
  Let $G=(V,E)$ be a finite directed acyclic planar graph. Fix sets of source vertices
  $\mathcal{X}\subset V$ and sink vertices $\mathcal{Y}\subset V$. Fix edge weights
  $w_{e}$ and a second set of weights $\tilde{w}_e$ for each directed edge $e\in E$. Fix
  functions $\psi_i\!:\mathcal{X}\to \R$ for $1\leq i\leq N$ and functions
  $\varphi_j\!:\mathcal{Y}\to \R$ for $1\leq j\leq N$ which satisfy the biorthogonality
  assumption with $\varphi_j^{(b)}$ defined via the $w_{e}$ weights. Write
  \[Z=Z\!\left(\mathcal{X},\mathcal{Y},\{w_e\}_{e\in E},
    \{\psi_i\}_{i=1}^{N},\{\varphi_j\}_{j=1}^{N}\right)\qand
  \widetilde{Z}~=~Z\!\left(\mathcal{X},\mathcal{Y},\{\tilde{w}_e\}_{e\in E},
    \{\psi_i\}_{i=1}^{N},\{\varphi_j\}_{j=1}^{N}\right).\] Then
  \begin{equation*}
    \frac{\widetilde{Z}}{Z} = \det(I-K + \widetilde{\mathcal{W}} \mathcal{W}^{-1} K)_{L^2(\mathcal{X})},
  \end{equation*}
  where $\widetilde{\mathcal{W}}\!:L^2(\mathcal{Y})\to L^2(\mathcal{X})$ is given by
  \eqref{eq:defW} with $w$ replaced by $\tilde w$, and $\mathcal{W}^{-1}K\!:L^2(\mathcal{X})\to
  L^2(\mathcal{Y})$ is defined in \eqref{eq:defWinvK}.
\end{theorem}

\begin{remark}
  As will be clear from the proof, the biorthogonality assumption implies that
  $Z\neq0$. Conversely, one can show that if $Z\neq0$ then there exists a linear change of
  basis in the space spanned by $\{\psi_i\}_{i=1}^{N}$ (or equally well the space spanned by
  $\{\varphi_j\}_{j=1}^{N}$) which leads back to the biorthogonality assumption being
  satisfied and does not change the ratio $\widetilde{Z}/Z$.
\end{remark}

Before turning to the proof of the theorem, let us check that it implies Theorem \ref{intropaththm}.

\begin{proof}[Proof of Theorem \ref{intropaththm}]
Recall that for the path-integral functional $f$, we defined a set of weights $\tilde{w}_e
= f_{n}(e) w_{e}$ where $e\in E_{n\to n+1}$. Let $Z=Z\!\left(V_0,V_T,\{w_e\}_{e\in
    E}, \{\psi_i\}_{i=1}^{N},\{\varphi_j\}_{j=1}^{N}\right)$ and $\widetilde
Z=Z\!\left(V_0,V_T,\{\tilde w_e\}_{e\in E}, \{\psi_i\}_{i=1}^{N},\{\varphi_j\}_{j=1}^{N}\right)$.
We claim that
\begin{equation*}
  \int\limits_{\Pi\in\mathcal{N\!.I.}(N;V_0\to
    V_T)}\!\!\!\!d\nu(\Pi)\,\prod_{i=1}^{N} f(\pi_i)  = \frac{\widetilde Z}{Z} = \det(I-K + \widetilde{\mathcal{W}} \mathcal{W}^{-1} K)_{L^2(V_0)}.
\end{equation*}
The second equality is an immediate corollary of Theorem \ref{combthm}. To see the first equality above observe that
\begin{eqnarray*}
  \frac{\widetilde{Z}}{Z} &=&  \sum_{\Pi\in \mathcal{N\!.I.}(N;V_0\to V_T)} \prod_{i=1}^{N}\prod_{n=0}^{T-1} f_n(\pi_i(n)\to \pi_i(n+1))\frac{W\!t(\Pi)}{Z} \\
  &=&\qquad\smashoperator{\int\limits_{\Pi\in\mathcal{N\!.I.}(N;V_0\to V_T)}}\quad
  d\nu(\Pi)\,\prod_{i=1}^{N}\prod_{n=0}^{T-1} f_n(\pi_i(n)\to \pi_i(n+1))
  =\qquad\smashoperator{\int\limits_{\Pi\in\mathcal{N\!.I.}(N;V_0\to V_T)}}\quad d\nu(\Pi) \, \prod_{i=1}^{N} f(\pi_i)
\end{eqnarray*}
as desired.
\end{proof}

\begin{proof}[Proof of Theorem \ref{combthm}]
  The proof is linear algebra. We may rewrite $\widetilde{Z}$ by first summing over the
  subsets of $\mathcal{X}$ and $\mathcal{Y}$ which host the source and sink vertices,
  and then considering all non-intersecting paths between these sets. Thus
\begin{align*}
\widetilde{Z} &=
\quad\qquad\smashoperator{\sum_{\substack{X=\{x_1,x_2,\cdots,x_N\}\subset
          \mathcal{X}\\Y=\{y_1,y_2,\cdots,y_N\}\subset \mathcal{Y}}}}\qquad
    \det[\psi_{i}(x_j)]_{i,j=1}^{N} \left(\sum_{\Pi \in \mathcal{N\!.I.}(N;X\to Y)}
 \prod_{i=1}^{N} \tilde{w}(\pi_i)\right) \det[\varphi_{i}(y_j)]_{i,j=1}^{N}\\
    &= \quad\qquad\smashoperator{ \sum_{\substack{\{x_1,x_2,\cdots,x_N\}\subset
          \mathcal{X}\\\{y_1,y_2,\cdots,y_N\}\subset \mathcal{Y}}}}\qquad
    \det[\psi_{i}(x_j)]_{i,j=1}^{N} \left(N!
      \det\!\left[\widetilde{\mathcal{W}}(x_i,y_j)\right]_{i,j=1}^{N} \right)
    \det[\varphi_{i}(y_j)]_{i,j=1}^{N}.
  \end{align*}
The second line follows by an application of Lemma \ref{LGV}.

We may now apply the Cauchy-Binet identity twice. The first application is with respect to the summation in the $x$'s, and it yields
  \begin{equation*}
    \sum_{\{x_1,x_2,\cdots,x_N\}\subset \mathcal{X}}  \det[\psi_{i}(x_j)]_{i,j=1}^{N} \det\!\left[\widetilde{\mathcal{W}}(x_i,y_j)\right]_{i,j=1}^{N} = \det\!\left[ \sum_{x\in\mathcal{X}}   \psi_{i}(x)\widetilde{\mathcal{W}}(x,y_j) \right]_{i,j=1}^{N}.
  \end{equation*}
  The second application likewise is applied to the summation in $y$'s and yields that
  \begin{equation}\label{eq:wtZ}
    \widetilde{Z} = \det\!\left[ \sum_{\substack{x\in \mathcal{X}\\ y\in \mathcal{Y}}} \psi_{i}(x)\widetilde{\mathcal{W}}(x,y) \varphi_{j}(y) \right]_{i,j=1}^{N}.
  \end{equation}

  Observe that by the same argument we can obtain an analogous expression for $Z$ with
  $\widetilde{\mathcal{W}}$ replaced by $\mathcal{W}$. However, by the definition of
  $\varphi_j^{(b)}$ we find that
  \begin{equation*}
    Z = \det\!\left[ \sum_{x\in \mathcal{X}} \psi_{i}(x)\varphi_{j}^{(b)}(x) \right]_{i,j=1}^{N}.
  \end{equation*}
  By the biorthogonality assumption,
  \begin{equation*}
    \sum_{x\in \mathcal{X}} \psi_{i}(x)\varphi_{j}^{(b)}(y) = \left\langle \psi_i,\varphi_j^{(b)}\right\rangle_{L^2(\mathcal{X})} = \uno{i=j},
  \end{equation*}
  and hence $Z=1$. Now observe that we can rewrite the kernel in appearing in
  \eqref{eq:wtZ} as
  \begin{equation*}
    \sum_{\substack{x\in \mathcal{X}\\ y\in \mathcal{Y}}} \psi_{i}(x)\widetilde{\mathcal{W}}(x,y) \varphi_{j}(y)  = \left\langle \widetilde{\mathcal{W}} \varphi_i,\psi_j\right\rangle_{L^2(\mathcal{X})}
  \end{equation*}
  by using the definition of the operator $\widetilde{\mathcal{W}}$ and the inner product
  on $L^2(\mathcal{X})$.

  Therefore, it remains to prove that
  \begin{equation}\label{remaintoprove}
    \det(I-K + \widetilde{\mathcal{W}} \mathcal{W}^{-1} K)_{L^2(\mathcal{X})}  =\det\!\left[  \left\langle \widetilde{\mathcal{W}} \varphi_i,\psi_j\right\rangle_{L^2(\mathcal{X})} \right]_{i,j=1}^{N}.
  \end{equation}

  To prove the above statement we will write down the matrix for the operator
  $I-K+\widetilde{\mathcal{W}}\mathcal{W}^{-1}K$ in the basis $(\varphi_1^{(b)},\ldots,
  \varphi_N^{(b)},(\textrm{span}\{\psi_i\}_{i=1}^{N})^{\bot})$, where
  $(\textrm{span}\{\psi_i\}_{i=1}^{N})^{\bot}$ represents any basis of the orthogonal
  complement of span$\{\psi_i\}_{i=1}^{N}$ in $L^2(\mathcal{X})$. Let us consider the
  action of the operator $K$ on the basis elements. On $\varphi_j^{(b)}$ one sees that $K$
  acts as the identity operator:
  \begin{eqnarray*}
    (K\varphi_j^{(b)})(x_1) &=& \sum_{x_2\in \mathcal{X}} \left(\sum_{i=1}^{N} \varphi_i^{(b)}(x_1) \psi_i(x_2)\right)\varphi_j^{(b)}(x_2) = \sum_{i=1}^{N} \varphi_i^{(b)}(x_1)  \left( \sum_{x_2\in \mathcal{X}}\psi_i(x_2)\varphi_j^{(b)}(x_2)\right) \\
    &=&  \sum_{i=1}^{N} \varphi_i^{(b)}(x_1) \left\langle \psi_i,\varphi_j^{(b)}\right\rangle_{L^2(\mathcal{X})} = \sum_{i=1}^{N} \varphi_i^{(b)}(x_1)\uno{i=j} =\varphi_j^{(b)}(x_1).
  \end{eqnarray*}
  It is likewise clear that $K$ acts on the basis elements of
  $(\textrm{span}\{\psi_i\}_{i=1}^{N})^{\bot}$ by taking them all to zero. Thus we may
  write $K$ as the matrix \begin{equation*} K =\left(
      \begin{array}{c|c}
        I & 0 \\\hline
        0 & 0 \\
      \end{array}
    \right)
  \end{equation*}
  where the two blocks correspond to the basis elements $\{\varphi_j^{(b)}\}_{j=1}^{N}$
  and $(\textrm{span}\{\psi_i\}_{i=1}^{N})^{\bot}$. This shows that
  \begin{equation*}
    I-K =\left(
      \begin{array}{c|c}
        0 & 0 \\\hline
        0 & I \\
      \end{array}
    \right).
  \end{equation*}

  The remaining operator to study is $ \widetilde{\mathcal{W}} \mathcal{W}^{-1} K$. 
  Writing the corresponding matrix in blocks as above we get from \eqref{eq:defWinvK} that
  \begin{equation*}
    \widetilde{\mathcal{W}} \mathcal{W}^{-1} K =\left(
      \begin{array}{c|c}
        A & 0 \\\hline
        * & 0 \\
      \end{array}
    \right)
\end{equation*}
where the $N\times N$ matrix $A$ is yet to be determined. The value of the star is not important. To see this, write
\begin{equation*}
    I-K+\widetilde{\mathcal{W}} \mathcal{W}^{-1} K =\left(
      \begin{array}{c|c}
        A & 0 \\\hline
        * & I \\
      \end{array}
    \right)
\end{equation*}
and observe then that
\begin{equation}\label{detA}
    \det(I-K + \widetilde{\mathcal{W}} \mathcal{W}^{-1} K)_{L^2(\mathcal{X})} = \det\!\left[A_{i,j}\right]_{i,j=1}^{N}.
\end{equation}
The value of $A_{i,j}$ can be found by using the inner product,
\begin{equation*}
    A_{i,j}  = \left\langle \widetilde{\mathcal{W}} \mathcal{W}^{-1} K \varphi_i^{(b)},\psi_j\right\rangle_{L^2(\mathcal{X})}.
\end{equation*}
Recalling that $\mathcal{W}^{-1} K \varphi_i^{(b)}= \varphi_i$ we deduce that
\begin{equation*}
    A_{i,j}  = \left\langle \widetilde{\mathcal{W}} \varphi_i,\psi_j\right\rangle_{L^2(\mathcal{X})}.
\end{equation*}
Combining this with (\ref{detA}) proves (\ref{remaintoprove}) and hence completes the
proof of the theorem.
\end{proof}

\section{Equivalence of extended kernel and path-integral kernel Fredholm determinants}\label{equivsec}
There are various types of limits one can take of graph-based non-intersecting line
ensembles. In this section we will show that formulas of the type given in the previous
section survive these limits. We do not prove this directly via a limit transition, but
rather show how such formulas arise via manipulations of the extended kernel Fredholm
determinants which describe these limiting systems. The main result of this section is,
therefore, the equality of two types of Fredholm determinants.

This equality will be stated in an abstract setting in this
section, and later applied to the examples we are interested in in Section
\ref{sec:examples}. A concrete example to keep in mind is the Airy$_2$ process, which we
will use throughout this section to illustrate the objects we will introduce and the
assumptions we will make on them. The Airy$_2$ process was introduced in the Introduction
and is discussed in further detail in Section \ref{Airy2Sec}, we refer the reader there
for details and just recall the definitions of the Airy Hamiltonian $H=-\Delta+x$ and the
Airy kernel $\KAirytwo(x,y)=\int_0^{\infty} d\lambda\Ai(x+\lambda)\Ai(y+\lambda)$.

\subsection{Fredholm determinants}\label{Fredsec}
Let us briefly introduce some of the basic notions related to Fredholm determinants (we
refer the reader to \cite{Simon} for more details). Consider a separable Hilbert space
$\ch$ and let $A$ be a bounded linear operator acting on $\ch$ ($\ch=L^2(\rr)$ in the
Airy$_2$ case). Let $|A|=\sqrt{A^*A}$ be the unique positive square root of the operator
$A^*A$. The \emph{trace norm} of $A$ is defined as $\|A\|_1= \sum_{n=1}^{\infty}\langle
e_n,|A|e_n\rangle$, where $\{e_n\}_{n\geq 1}$ is any orthonormal basis of $\ch$. We say
that $A\in\mathcal{B}_1(\ch)$, the family of \emph{trace class operators}, if
$\|A\|_1<\infty$. For $A\in\mathcal{B}_1(\ch)$, one can define the trace
$\tr(A)=\sum_{n=1}^{\infty}\langle e_n,A e_n\rangle$. For later use we also define the
\emph{Hilbert-Schmidt norm} $\|A\|_2 =\sqrt{\tr(|A|^2)}$ and say that $A\in\cb_2(\ch)$,
the family \emph{Hilbert-Schmidt operators}, if $\|A\|_2<\infty$. Given $A\in\cb_1(\ch)$
one can define a generalization of the finite-dimensional determinant, the \emph{Fredholm
  determinant} $\det(I+A)_{\ch}$. We refer the reader to \cite{Simon} for the details of
the definition in this level of generality and just point out that, as expected,
$\det(I+A)_\ch=\prod_n(1+\lambda_n)$, where $\lambda_n$ are the eigenvalues of $A$
(counted with algebraic multiplicity).

The result presented in this section (Theorem \ref{thm:extendedToBVP}) can be stated,
under some conditions, for operators acting on a general separable Hilbert
space. Nevertheless, in order to keep the presentation as simple as possible, and since it
is the setting we need for the examples in Section \ref{sec:examples}, we will restrict
ourselves to the case of integral operators on an $L^2$ space.

More precisely, we assume we are given a measure space $(X,\Sigma,\mu)$ and consider the
Hilbert space $L^2(X,\mu)$. For brevity we will drop $\mu$ from the notation. We will also
denote by $\cm(X)$ the space of real-valued measurable functions on $X$. By an
\emph{integral operator} we mean an operator $A\!:D\subseteq\cm(X)\longrightarrow\cm(X)$
acting as $Af(x)=\int_X \mu(dy)\,A(x,y)f(y)$, where $A\!:X\!\times\!X\longrightarrow\rr$
is the \emph{integral kernel} of $A$. We will often speak interchangeably of an integral
operator and its kernel. In particular we have abused notation by using the same
letter to denote an integral operator and its kernel. We recall that the product of two integral operators is defined by
$AB(x,y)=\int_X\mu(dz)\,A(x,z)B(z,y)$.

Though we will not appeal to this, we note that the Fredholm determinant $\det(I-K)_{L^2(X)}$ of a trace class operator $K\!:L^2(X)\to
L^2(X)$ with continuous (in both $x$ and $y$) integral kernel $K(x,y)$ has the following (absolutely convergent) series expansion
\begin{equation}
\det(I-K)_{L^2(X)} = 1+ \sum_{k\geq 1} \frac{(-1)^k}{k!} \int_X d\mu(x_1)\cdots \int_X d\mu(x_k) \det\!\left[K(x_i,x_j)\right]_{i,j=1}^{k}.\label{eq:fredholm}
\end{equation}



\subsection{Assumptions for the theorem}
In order to state the main theorem of this section in a fairly broad context, we must introduce a few operators and impose certain assumptions upon them. Most of the assumptions are technical and intended to ensure well-definedness or finiteness of the various quantities involved in the statement of the theorem. The main (not just technical) assumption is given in Assumption \ref{assum:1}.

Fix $t_1<\dots<t_n$ for the duration of this section. We will be interested in comparing the
Fredholm determinant of certain integral operators acting on the Hilbert spaces $L^2(X)$
and $L^2(\{t_1,\dotsc,t_n\}\!\times\!X)$ (the measure we use in the second space is the
product of the counting measure on $\{t_1,\dotsc,t_n\}$ and $\mu$). The operators we consider will be constructed from the following four families of operators:
\begin{itemize}
\item For each $1\leq i<j\leq n$, an integral kernel $\cw_{t_i,t_j}$ (for convenience we also introduce the notation $\cw_{t_i,t_i}=I$);
\item For each $1\leq i\leq n$, an integral kernel $K_{t_i}$;
\item For each $1\leq i<j\leq n$, an integral kernel $\cw_{t_j,t_i}K_{t_i}$ (for convenience we also introduce the notation $\cw_{t_i,t_i}K_{t_i}=K_{t_i}$);
\item For each $1\leq i\leq n$, a multiplication operator $Q_{t_i}$ acting on $\cm(X)$
  as $Q_{t_i}f(x)=q_{t_i}(x)f(x)$ for some $q_{t_i}\in\cm(X)$.
\end{itemize}
The reason for the choice of notation $\cw_{t_j,t_i}K_{t_i}$ in the third family of
operators is that we will assume below that $\cw_{t_i,t_j}\cw_{t_j,t_i}K_{t_i}=K_{t_i}$
for $i<j$ (so that even though it is not defined above as its own operator, $\cw_{t_j,t_i}$ can be thought of as a right inverse of $\cw_{t_i,t_j}$ on the range of $K_{t_i}$).

We make the following (technical) assumption.

\begin{assumption}\label{assum:0}
  \mbox{}
  \begin{enumerate}[label=(\roman*)]
  \item The integral operators $Q_{t_i}\cw_{t_i,t_j}$, $Q_{t_i}K_{t_i}$,
  $Q_{t_i}\cw_{t_i,t_j}K_{t_j}$ and $Q_{t_j}\cw_{t_j,t_i}K_{t_i}$ for $1\leq i<j\leq n$
  are all bounded operators mapping $L^2(X)$ to itself.
  \item The operator
  \[K_{t_1}-\oQ_{t_1}\cw_{t_1,t_2}\oQ_{t_2}\dotsm\cw_{t_{n-1},t_n}\oQ_{t_n}
  \cw_{t_n,t_1}K_{t_1},\] where $\oQ_{t_i}=I-Q_{t_i}$, is a bounded operator mapping
  $L^2(X)$ to itself.
  \end{enumerate}
\end{assumption}


The last operator in the assumption will appear in the formula provided in Theorem
\ref{thm:extendedToBVP}. An alternative expression for this operator, which is in some
cases more convenient for checking the assumption, is given in Lemma \ref{lem:alt}.

 In the case of the Airy$_2$ process we take $X=\rr$, choose $\mu$ to be the Lebesgue
 measure and set $\cw_{t_i,t_j}=e^{(t_i-t_j)H}$, $K_{t_i}=\KAirytwo$, and
 $\cw_{t_j,t_i}K_{t_i}=e^{(t_j-t_i)H}\KAirytwo$ for $1\leq i\leq j\leq n$. One can take
 for example the operators $Q_{t_i}$ to be projections on intervals $[a_i,\infty)$, that
 is, $Q_{t_i}f(x)=\uno{x\geq a_i}f(x)$, which corresponds to studying the finite
 dimensional distributions of the Airy$_2$ process (we will make a more general choice in
 Section \ref{Airy2Sec}).

Going back to the general setting, we will make a certain algebraic assumption on the
operators $\cw_{t_i,t_j}$, $K_{t_i}$ and $\cw_{t_j,t_i}K_{t_i}$.

\begin{assumption}\label{assum:1}
  For each $i\leq j\leq k$ the following hold:
    \begin{enumerate}[label=(\roman*)]
    \item \emph{Right-invertibility}: $\cw_{t_i,t_j}\cw_{t_j,t_i}K_{t_i}=K_{t_i}$;
    \item \emph{Semigroup property}: $\cw_{t_i,t_j}\cw_{t_j,t_k}=\cw_{t_i,t_k}$;
    \item \emph{Reversibility relation}: $\cw_{t_i,t_j}K_{t_j}=K_{t_i}\cw_{t_i,t_j}$.
    \end{enumerate}
 \end{assumption}

 The second property is clear in the Airy$_2$ case, while (i) and (iii) follow from the
 fact that $\KAirytwo$ is the projection operator into the negative (generalized)
 eigenspace of the Airy Hamiltonian $H$ (see Section \ref{Airy2Sec}).

 Let us now explain how these operators will be used. Using the kernels introduced above
 we define an \emph{extended kernel} $K^{\rm ext}$ as follows: for $1\leq i,j\leq n$ and
 $x,y\in X$,
\begin{equation}\label{eq:generalExt}
  K^{\rm ext}(t_i,x;t_j,y)=
  \begin{dcases*}
    \cw_{t_i,t_j}K_{t_j}(x,y) & if $i\geq j$,\\
    -\cw_{t_i,t_j}(I-K_{t_j})(x,y) & if $i<j$.
  \end{dcases*}
\end{equation}
This definition coincides with the usual notion of extended correlation kernels of
determinantal point processes, cf. \cite{tracyWidom-Diff,BR,Eyn,PNGJ,Nag,borDet}. In the
case of the Airy$_2$ process, it coincides with the definition given in the Introduction
and in \eqref{eq:KAi2}. As an operator, $K^{\rm ext}$ acts on
$f\in L^1_{\rm loc}(\{t_1,\dotsc,t_n\}\!\times\!X)$ as
\[K^{\rm ext}f(t_i,x)=\sum_{j=1}^n\int_Xd\mu(y)\,K^{\rm ext}(t_i,x;t_j,y)f(t_j,y).\] See
Section \ref{sec:examples} for concrete examples.


 We also need to make the following (technical) analytical assumption.

\begin{assumption}\label{assum:2}
  One can choose multiplication operators $V_{t_i}$, $V'_{t_i}$, $U_{t_i}$ and $U'_{t_i}$
  acting on $\cm(X)$, for $1\leq i\leq n$, in such a way that:
  \begin{enumerate}[label=(\roman*)]
  \item $V_{t_i}'V_{t_i}Q_{t_i}=Q_{t_i}$  and $K_{t_i}U_{t_i}'U_{t_i}=K_{t_i}$, for all $1\leq i\leq n$.
  \item The operators $V_{t_i}Q_{t_i}K_{t_i}V_{t_i}'$, $V_{t_i}Q_{t_i}\cw_{t_i,t_j}V_{t_j}'$,
    $V_{t_i}Q_{t_i}\cw_{t_i,t_j}K_{t_j}V_{t_j}'$ and
    $V_{t_j}Q_{t_j}\cw_{t_j,t_i}K_{t_i}V_{t_i}'$ preserve $L^2(X)$ and are trace class in $L^2(X)$, for all $1\leq i<j\leq n$.
  \item The operator
    $U_{t_{i}}\!\left[\cw_{t_{i},t_1}K_{t_1}-\oQ_{t_{i}}\cw_{t_{i},t_{i+1}}\dotsm
      \oQ_{t_{n-1}}\cw_{t_{{n-1}},t_{n}}\oQ_{t_{n}}\cw_{t_{n},t_1}K_{t_1}\right]U_{t_1}'$ preserves
    $L^2(X)$ and is trace class in $L^2(X)$, for all $1\leq i\leq n$, where
    $\oQ_{t_i}=I-Q_{t_i}$.
  \end{enumerate}
\end{assumption}


The primes in $U'_{t_i}$ and $V'_{t_i}$ mean that these are almost (left) inverses of the
operators $U_{t_i}$ and $V_{t_i}$, and hence the multiplication by these operators in (ii)
and (iii) should be thought of as a conjugation. The distinction is because in many cases
it will be necessary to let $V_{t_i}$ be multiplication by a function which is 0 where
$q_{t_i}$ is 0, in which case $V_{t_i}$ is not invertible, with an analogous situation for
$U_{t_i}$ and $U_{t_i}'$.

Before stating the main result of this section, Theorem \ref{thm:extendedToBVP}, let us
state a formula which reexpresses the operator appearing in Assumption
\ref{assum:2}(iii). Besides being used in the proof of the below theorem,
this formula is often useful in checking the assumption (for example, as in Remark \ref{rem:assum2}).

\begin{lemma}\label{lem:alt}
  Writing $\oQ_{t}=I-Q_{t}$, we have, for any $1\leq i\leq n$,
  \begin{multline*}
    \cw_{t_i,t_1}K_{t_1}-\oQ_{t_i}\cw_{t_i,t_2}\oQ_{t_2}\dotsm\cw_{t_{n-1},t_n}\oQ_{t_n}
    \cw_{t_n,t_1}K_{t_1}\\
    =\sum_{j=i}^n\sum_{k=0}^{n-j}(-1)^{k}\quad\smashoperator{\sum_{j=a_0<a_1<\dots<a_k\leq
        n}}\quad\cw_{t_i,t_j}Q_{t_j}\cw_{t_{j},{t_{a_1}}}Q_{t_{a_1}}\cw_{t_{a_1},t_{a_2}}
    Q_{t_{a_{k-1}}}\cw_{t_{a_{k-1}},t_{a_k}}Q_{t_{a_k}}\cw_{t_{a_k},t_{1}}K_{t_1}.
  \end{multline*}
\end{lemma}

We postpone the proof of this lemma until the end of this section.

\begin{remark}\label{rem:assum2}
  Suppose that there exist multiplication operators $\wt V_{t_i}$ and $\wt V'_{t_i}$
  acting on $\cm(X)$, for $1\leq i\leq n$, in such a way that:
  \begin{enumerate}[label=(\roman*)]
  \item $\wt V_{t_i}'\wt V_{t_i}Q_{t_i}=Q_{t_i}$ and $K_{t_i}\wt V_{t_i}\wt
    V_{t_i}'=K_{t_i}$, for all $1\leq i\leq n$;
  \item The operators $\wt V_{t_i}Q_{t_i}K_{t_i}\wt V_{t_i}'$, $\wt
    V_{t_i}Q_{t_i}\cw_{t_i,t_j}\wt V_{t_j}'$, $\wt V_{t_i}Q_{t_i}\cw_{t_i,t_j}K_{t_j}\wt
    V_{t_j}'$ and $\wt V_{t_j}Q_{t_j}\cw_{t_j,t_i}K_{t_i}\wt V_{t_i}'$ preserve $L^2(X)$
    and are trace class in $L^2(X)$, for all $1\leq i<j\leq n$.
  \end{enumerate}
  Then it is not hard to check, using the formula given in Lemma \ref{lem:alt}, that
  Assumption \ref{assum:2} holds, taking $U_{t_i}=V_{t_i}'=\wt V_{t_i}'$ and
  $U_{t_i}'=V_{t_i}=\wt V_{t_i}$ (see the end of the proof of Corollary
  \ref{Airytwobddform} in Appendix \ref{sec:app} for more details). In the case of the
  Airy$_2$ process, when the operators $Q_{t_i}$ are of the form $Q_{t_i}f(x)=\uno{x\geq
    a_i}f(x)$ as discussed above, both $\wt V_{t_i}$ and $\wt V_{t_i}'$ can be taken
  to be the identity. If, on the other hand, one assumes $q_{t_i}(x)$ to be 0 for
  $x<a_i$ but to grow at a certain rate for $x\geq a_i$, as we will in Section
  \ref{Airy2Sec}, then it is necessary to choose these operators more carefully (see the
  proof of Corollary \ref{Airytwobddform}).
\end{remark}

\subsection{Identity between extended and  path-integral kernel Fredholm determinants}
Define a diagonal operator $Q$ acting on $f\in\cm(\{t_1,\dotsc,t_n\}\!\times\!X)$ as
\begin{equation}\label{eq:defQ}
  Qf(t_i,\cdot)=Q_{t_i}f(t_i,\cdot).
\end{equation}
Note that, by Assumption \ref{assum:0}, $QK^{\rm ext}$ preserves
$L^2(\{t_1,\dotsc,t_n\}\!\times\!X)$. The following result expresses the Fredholm
determinant of $I-QK^{\rm ext}$ on $L^2(\{t_1,\dotsc,t_n\}\!\times\!X)$ as a Fredholm
determinant on $L^2(X)$. The first example of such a formula was provided by \cite{PS} for
the case of the Airy$_2$ process (see also \cite{prolhacSpohn}). This was later extended
to the Airy$_1$ process in \cite{quastelRemAiry1}. This type of formulas have recently
been found to be very useful in the study of these processes, see for example
\cite{CQR,MQR,quastelRemAiry1,qr12,qrt}.

\begin{theorem}\label{thm:extendedToBVP}
  With the above notation, and under Assumptions \ref{assum:0}, \ref{assum:1} and \ref{assum:2}, we have
  \begin{equation}\label{eq:extToBVP}
    \det\!\big(I-QK^{\rm ext}\big)_{L^2(\{t_1,\dots,t_n\}\times X)}
    =\det\!\big(I-K_{t_1}+\oQ_{t_1}\cw_{t_1,t_2}\oQ_{t_2}\dotsm\cw_{t_{n-1},t_n}\oQ_{t_n}
    \cw_{t_n,t_1}K_{t_1}\big)_{L^2(X)},
  \end{equation}
  where $\oQ_t=I-Q_t$.
\end{theorem}

\begin{remark}\label{rem:extToBVP}
  The operators appearing in both Fredholm determinants preserve $L^2(X)$ by Assumption
  \ref{assum:0}. Moreover, the Fredholm determinants are well-defined thanks to
  Assumption \ref{assum:2}, even though the operators appearing there are not necessarily
  trace class. In fact, if we define the diagonal operator $V$ acting on
  $u\in L^2(\{t_1,\dotsc,t_n\}\!\times\!X)$ as $(Vu)_{t_i}=V_{t_i}u_{t_i}$, and similarly
  define $V'$, then $VQK^{\rm ext}V'$ is trace class by Assumption \ref{assum:2}(ii) and
  by the cyclic property of the determinant and the fact that $V'VQ=Q$ it leads to the
  same Fredholm expansion for $\det(I-VQK^{\rm
    ext}V')_{L^2(\{t_1,\dots,t_n\}\times X)}$ and for $\det(I-QK^{\rm
    ext})_{L^2(\{t_1,\dots,t_n\}\times X)}$. The same argument applies to the Fredholm
  determinant on the right-hand side of \eqref{eq:extToBVP} by Assumption \ref{assum:2}(iii),
  if we multiply it on the left by $U_{t_1}$ and on the right by $U_{t_1}'$. Hence
  both sides of \eqref{eq:extToBVP} are well-defined and one should really read the equality as
  \begin{multline*}
    \det\!\big(I-VQK^{\rm ext}V'\big)_{L^2(\{t_1,\dots,t_n\}\times X)}\\
    =\det\!\big(I-U_{t_1}(K_{t_1}-\oQ_{t_1}\cw_{t_1,t_2}\dotsm\cw_{t_{n-1},t_n}\oQ_{t_n}
    \cw_{t_n,t_1}K_{t_1})U_{t_1}'\big)_{L^2(X)}.
  \end{multline*}
\end{remark}


\begin{proof}[Proof of Theorem \ref{thm:extendedToBVP}]
  The proof of this result is a generalization of the proof of Theorem 1 of
  \cite{quastelRemAiry1} (see also the Appendix of \cite{prolhacSpohn}). We will retain
  most of the notation of \cite{quastelRemAiry1,prolhacSpohn}, and as in those papers we
  use sans-serif fonts (e.g. $\sW$) for operators on
  $L^2(\{t_1,\dots,t_n\}\!\times\!X)$. This space can be identified with the space
  $\bigoplus_{t\in\{t_1,\dots,t_n\}}L^2(X)$, and hence we may (and will) think of an
  operator $\sW$ on $L^2(\{t_1,\dots,t_n\}\!\times\!X)$ as an operator-valued $n\times n$
  matrix. We will use serif fonts for the matrix entries (e.g. $\sW_{i,j}=W$ for some $W$
  acting on $L^2(X)$). All determinants throughout this proof are computed on
  $L^2(\{t_1,\dots,t_n\}\!\times\!X)$ unless otherwise indicated.

  We will use repeatedly the following facts about trace class operators and Fredholm
  determinants on a separable Hilbert space $\ch$:
  \begin{enumerate}[label=(\roman*)]
  \item If $A,B\in\cb_1(\ch)$ then $AB\in\cb_1(\ch)$ and
    \[~~\quad\quad\det((I+A)(I+B))_{\ch}=\det(I+A)_{\ch}\det(I+B)_{\ch}.\]
    Moreover, if $A$ and $B$ are bounded linear operators on $\ch$ and both $AB,BA\in
    \cb_1(\ch)$ then
    \begin{equation}\label{eq:cyclic}
      \det(I+AB)_{\ch}=\det(I+BA)_{\ch}.
    \end{equation}
  \item An operator acting on $\bigoplus_{t\in\{t_1,\dots,t_n\}}\ch$ is trace class if and
    only if all of its matrix entries are trace class.
  \end{enumerate}
  To simplify notation throughout the proof we will replace subscripts of the form $t_i$
  by $i$, so for example $\cw_{i,j}=\cw_{t_i,t_j}$.

  Recall that we are assuming $t_1<t_2<\dots<t_n$. Let $\sK=QK^{\rm ext}$. Then $\sK$ can
  be written as
  \begin{equation}\label{eq:sK}
    \sK=\sQ(\sW^{-}\sK^{\rm d}+\sW^{+}(\sK^{\rm d}-\sI)),
  \end{equation}
  where
  \begin{equation*}
    \sK^{\rm d}_{ij}=K_{i}\uno{i=j},\qquad \sQ_{i,j}=Q_{{i}}\uno{i=j}
  \end{equation*}
  and $\sW^{-}$, $\sW^{+}$ are lower triangular, respectively strictly upper triangular,
  and defined by
  \begin{equation*}
    \sW^{-}_{ij} = \cw_{i,j}\uno{i\geq j},\quad
    \sW^{+}_{ij}=\cw_{{i},{j}}\uno{i < j}.
  \end{equation*}
  Here we are slightly abusing notation, because $\cw_{i,j}$ is not defined for
  $i>j$. However, since $\sW^-$ appears applied after $\sK^{\rm d}$, the formula makes
  sense, with $[\sW^{-}\sK^{\rm d}]_{i,j}=\cw_{i,j}K_{j}$ for $i>j$. We also define
  the diagonal operators $\sV$, $\sV'$, $\sU$ and $\sU'$ by
  \[\sV_{i,j}=V_{i}\uno{i=j},\qquad\sV_{i,j}'=V_{i}'\uno{i=j},\qquad\sU_{i,j}=U_{i}\uno{i=j}
  \qqand\sU'_{i,j}=U'_{i}\uno{i=j}.\]

  In order to manipulate the Fredholm determinant of $\sI-\sK$ we will need to make sure
  at each step that the appropriate operators preserve $L^2(X)$ and are trace class in
  $L^2(X)$ as needed. As a consequence, the proof is slightly cumbersome, so we will first
  briefly explain the main idea, ignoring some details and all analytical
  issues.

  Our goal is to manipulate the determinant of $\sI-\sK$ in such a way that we end up with
  the determinant of an operator-valued matrix $\sI-\wt\sK$ where only the first column of
  $\wt\sK$ is non-zero. If we achieve this, then we will have
  $\det(\sI-\sK)=\det(\sI-\wt\sK)=\det(\sI-\wt\sK_{1,1})_{L^2(X)}$, and all we will need
  to do is compute $\sK_{1,1}$. The key to obtain such an identity is the following
  observation. Using the semigroup property in Assumption \ref{assum:1}(ii) one can check
  directly that
  \begin{equation}\label{eq:IT+}
    \big[(\sI+\sW^{+})^{-1})\big]_{i,j}=I\uno{j=i}-\cw_{{i},{i+1}}\uno{j=i+1}.
  \end{equation}
  This identity is meant in the sense of products of integral kernels, where the product
  of the identity operator with an integral kernel is defined in the obvious way. Now
  using the identity $\cw_{i,{j-1}}K_{j-1}\cw_{{j-1},j}=\cw_{i,j}K_{j}$ from Assumptions
  \ref{assum:1}(ii) and \ref{assum:1}(iii) we get that
  \begin{equation}
      \big[(\sW^{-}+\sW^{+})\sK^{\rm d}(\sI+\sW^{+})^{-1}\big]_{i,j}
      =\cw_{i,j}K_{j}-\cw_{i,{j-1}}K_{j-1}\cw_{{j-1},j}\uno{j>1}
      =\cw_{{i},{1}}K_{1}\uno{j=1}.\label{eq:T-T+}
  \end{equation}
  Note that only the first column of this matrix has non-zero entries. To take advantage
  of this fact we rewrite $\sK$ as
  \begin{equation}
    \sK=\sQ(\sW^-+\sW^+)\sK^{\rm d}(\sI+\sW^+)^{-1}(\sI+\sW^+)-\sQ\sW^+,\label{eq:sKe}
  \end{equation}
  so that
  \[\sI-\sK=(\sI+\sQ\sW^+)\big[\sI-(\sI+\sQ\sW^+)^{-1}\sQ(\sW^-+\sW^+)\sK^{\rm
    d}(\sI+\sW^+)^{-1}(\sI+\sW^+)\big].\]
  The invertibility of $\sI+\sQ\sW^+$ follows from the fact that $\sQ\sW^+$ is strictly
  upper triangular. This fact also implies that $\det(\sI+\sQ\sW^+)=1$, and hence
  \begin{align*}
    \det(\sI-\sK)&=\det(\sI-(\sI+\sQ\sW^+)^{-1}\sQ(\sW^-+\sW^+)\sK^{\rm
      d}(\sI+\sW^+)^{-1}(\sI+\sW^+))\\
    &=\det(\sI-(\sI+\sW^+)(\sI+\sQ\sW^+)^{-1}\sQ(\sW^-+\sW^+)\sK^{\rm
    d}(\sI+\sW^+)^{-1}),
  \end{align*}
  where we have used the cyclic property of the determinant. Recalling that only the first
  column of $(\sW^-+\sW^+)\sK^{\rm d}(\sI+\sW^+)^{-1}$ is non-zero we deduce that
  \[\wt\sK=(\sI+\sW^+)(\sI+\sQ\sW^+)^{-1}\sQ(\sW^-+\sW^+)\sK^{\rm
    d}(\sI+\sW^+)^{-1}\] has the same property and hence
  $\det(\sI-\sK)=\det(\sI-\wt\sK)=\det(\sI-\wt\sK_{1,1})_{L^2(X)}$ as desired.

  The rest of the proof will consist in making the above argument rigorous and precise and
  then computing the resulting $\wt\sK_{1,1}$.  Recall that, by Assumption
  \ref{assum:2}(ii), each entry in the operator-valued matrix $\sV\sK\sV'$ is trace class
  in $L^2(X)$. Let
  \begin{equation}
    \sWo=\sV\sQ\sW^+\sV'\qqand\sWt=\sV\sQ(\sW^-+\sW^+)\sK^{\rm d}\sV'.\label{eq:sWs}
  \end{equation}
  Since $\sV\sQ\sW^+\sV'$ is strictly upper triangular, we have
  $(\sV\sQ\sW^+\sV')^{n+1}=0$, so $\sI+\sWo$ is invertible:
  \begin{equation}
    (\sI+\sWo)^{-1}=\sum_{k=0}^n(-1)^k(\sV\sQ\sW^{+}\sV')^k.\label{eq:inv}
  \end{equation}
  Therefore we can write
  \begin{equation*}
    \det\!\big(\sI-\sV\sK\sV'\big) =\det\!\big((\sI+\sWo)(\sI-(\sI+\sWo)^{-1}\sWt)\big).
  \end{equation*}
  We remark that $\sWo$, $\sWt$ and $(\sI+\sWo)^{-1}$ are trace class in $L^2(X)$ by
  Assumption \ref{assum:2}(ii) and \eqref{eq:inv}, and thus from the last identity we deduce that
  \begin{equation}
    \det\!\big(\sI-\sV\sK\sV'\big)=\det\!\big(\sI+\sWo\big)\det\!\big(\sI-(\sI+\sWo)^{-1}\sWt\big)
    =\det\!\big(\sI-(\sI+\sWo)^{-1}\sWt\big),\label{eq:detRed}
  \end{equation}
  where the second equality follows from the fact that, since $\sWo$ is strictly upper
  triangular, its only eigenvalue is 0, so $\det(\sI+\sWo)=1$.


  Write
  \begin{equation}
    (\sI+\sWo)^{-1}\sWt=\tsWo\tsWt\label{eq:W3W4}
  \end{equation}
  with
  \[\tsWo=(\sI+\sWo)^{-1}\sV\sQ(\sW^-+\sW^+)\sK^{\rm d}(\sI+\sW^+)^{-1}\sU'\qqand
  \tsWt=\sU(\sI+\sW^+)\sV'.\] Here we are using \eqref{eq:T-T+} and the identity
  $\sK\sU\sU'=\sK$. We have already checked that $\tsWo\tsWt$ is trace class in $L^2(X)$.
  Thus if we prove that $\tsWt\tsWo$ is also trace class we can deduce from
  \eqref{eq:cyclic}, \eqref{eq:detRed} and \eqref{eq:W3W4} that
  \begin{equation}
    \label{eq:detFinalExt}
    \det\!\big(\sI-\sV\sK\sV'\big)=\det\!\big(\sI-\tsWt\tsWo\big).
  \end{equation}

  We want to obtain an explicit expression for the kernel $\tsWt\tsWo$. Note that, in
  view of \eqref{eq:inv} and the fact that $\sV'\sV\sQ=\sQ$,
  $\sV'(\sI+\sWo)^{-1}\sV\sQ=(\sI+\sQ\sW^+)^{-1}\sQ$, so all the factors $\sV$ and $\sV'$
  cancel in $\tsWt\tsWo$:
   \[\tsWt\tsWo=\sU(\sI+\sW^+)(\sI+\sQ\sW^+)^{-1}\sQ(\sW^-+\sW^+)\sK^{\rm
    d}(\sI+\sW^+)^{-1}\sU'.\] From \eqref{eq:T-T+} and
  the semigroup property we deduce that, for $0\leq k\leq n-i$,
  \begin{multline*}
    \left[(\sQ\sW^+)^k\sQ(\sW^-+\sW^+)\sK^{\rm d}(\sI+\sW^+)^{-1}\right]_{i,1}\\
    =\qquad\smashoperator{\sum_{i<a_1<\dots<a_k\leq n}}\qquad
    Q_{{i}}\cw_{i,{a_1}}Q_{{a_1}}\cw_{{a_{1}},{a_2}} \dotsm
    Q_{{a_{k-1}}}\cw_{{a_{k-1}},{a_k}}Q_{{a_k}}\cw_{{a_k},{1}}K_{1},
  \end{multline*}
  while for $k>n-i$ the left-hand side above equals 0 (the case $k=0$ is interpreted as
  $Q_{i}\cw_{i,1}K_{1}$). Summing the above times $(-1)^k$ from $k=0$ to $k=n-i$ we get
  directly from the last formula and \eqref{eq:inv} that
  \begin{multline}
    \label{eq:preFinaldecsW}
    \left[(\sI+\sQ\sW^+)^{-1}\sQ(\sW^-+\sW^+)\sK^{\rm d}(\sI+\sW^+)^{-1}\right]_{i,j}\\
    =\uno{j=1}\left[Q_{i}\cw_{i,1}K_{1}+\sum_{k=1}^{n-i}(-1)^{k}\quad\smashoperator{\sum_{i<a_1<\dots<a_k\leq
        n}}\quad Q_{{i}}\cw_{i,{a_1}}Q_{{a_1}}\cw_{{a_{1}},{a_2}} \dotsm
    Q_{{a_{k-1}}}\cw_{{a_{k-1}},{a_k}}Q_{{a_k}}\cw_{{a_k},{1}}K_{1}\right].
  \end{multline}
  %
  %
  Note that only the first column of the above matrix contains non-zero entries. Since
  $\sU(\sI+\sW^+)$ is upper triangular and $\sU'$ is diagonal, the same is true for
  $\tsWt\tsWo$. Pre-multiplying \eqref{eq:preFinaldecsW} by $\sU(\sI+\sW^+)$ we get
  \begin{equation}
    \big(\tsWt\tsWo\big)_{i,1}
    =\sum_{j=i}^n\sum_{k=0}^{n-j}(-1)^{k}\quad\smashoperator{\sum_{j=a_0<a_1<\dots<a_k\leq
        n}}\quad U_{i}\cw_{i,j}Q_{j}\cw_{{j},{{a_1}}}Q_{{a_1}}\cw_{{a_1},{a_2}}
    Q_{{a_{k-1}}}\cw_{{a_{k-1}},{a_k}}Q_{{a_k}}\cw_{{a_k},{1}}K_{1}U_{1}'.\label{eq:tswti1}
  \end{equation}
  By Lemma \ref{lem:alt} we deduce that
  \begin{equation}
    \label{eq:finishDetProof}
    \big(\tsWt\tsWo\big)_{i,1}=U_i\big[\cw_{i,1}K_{1}-\oQ_{i}\cw_{i,2}\oQ_{2}\dotsm\cw_{{n-1},n}\oQ_{n}
    \cw_{n,1}K_{1}\big]U_1'.
  \end{equation}
    By Assumption \ref{assum:2}(iii) this operator is trace class, which provides the needed
  justification for writing \eqref{eq:detFinalExt}, and then since only the first column
  of $\tsWt\tsWo$ is non-zero we deduce that
  \[\det\!\big(\sI-\sV\sK\sV'\big)=\det\!\Big(I-\big(\tsWt\tsWo\big)_{1,1}\Big)_{L^2(\rr)}.\]
  Setting $i=1$ in \eqref{eq:finishDetProof} yields the result.
\end{proof}

In order to finish the proof of Theorem \ref{thm:extendedToBVP} it remains to prove Lemma
\ref{lem:alt}.

\begin{proof}[Proof of Lemma \ref{lem:alt}]
  We start with the right-hand side of the identity. Replace each $Q_{i}$ by $I-\oQ_{i}$
  except for the first one to get
  \begin{equation*}
    \sum_{j=i}^n\sum_{k=0}^{n-j}(-1)^{k}\sum_{m=0}^{k}\binom{n-j-m}{k-m}(-1)^m\quad\smashoperator{\sum_{j=b_0<b_1<\dots<b_m\leq n}}\quad
    \cw_{{i},{b_0}}Q_{{b_0}}\cw_{{b_0},{b_1}}
    \dotsm\oQ_{{b_{m-1}}}\cw_{{b_{m-1}},{b_m}}\oQ_{{b_m}}\cw_{{b_m},{1}}K_{1}
  \end{equation*}
  where, as in the above proof, we have written $i$ instead of $t_i$ in the subscripts.
  Interchanging the order of summation leads to
  $\sum_{j=i}^n\sum_{m=0}^{n-j}\sum_{k=m}^{n-j}(-1)^{k+m}\binom{n-j-m}{k-m}(\star)$, where
  $(\star)$ represents the last sum above, and is independent of $k$. Noting that
  $\sum_{k=m}^{n-j}\binom{n-j-m}{k-m}(-1)^{k+m}=\uno{m=n-j}$, the above expression can be
  rewritten as
  \begin{align*}
    \sum_{j=i}^{n}\qquad&\quad\,\smashoperator{\sum_{j=b_0<b_1<\dots<b_{n-j}\leq n}}
    \quad \cw_{{i},{b_0}}Q_{{b_0}}\cw_{b_0,b_1}\oQ_{b_1}\dotsm\cw_{{b_{m-1}},{b_{n-j}}}\oQ_{{b_{n-j}}}\cw_{{b_{n-j}},{1}}K_1\\
    &=\sum_{j=i}^n\cw_{{i},{j}}(I-\oQ_{{j}})\cw_{{j},{j+1}}\oQ_{{j+1}}
    \dotsm\cw_{b_{n-1},b_{n}}\oQ_{{b_{n}}}\cw_{n,{1}}K_{1}\\
    &=\sum_{j=i}^{n}\Big[\cw_{{i},{j+1}}\oQ_{{j+1}}\cw_{{j+1},{j+2}}\oQ_{{j+2}}
    \dotsm\cw_{{b_{n-1}},{b_{n}}}\oQ_{{b_{n}}}\cw_{n,{1}}K_{1}\\
    &\hspace{2.3in}-\cw_{{i},{j}}\oQ_{{j}}\cw_{{j},{j+1}}\oQ_{{j+1}}
    \dotsm\cw_{{b_{n-1}},{b_{n}}}\oQ_{{b_{n}}}\cw_{n,{1}}K_{1}\Big]\\
    &=\cw_{i,1}K_{1}-\oQ_{i}\cw_{i,2}\oQ_{2}\dotsm\cw_{{n-1},n}\oQ_{n}
    \cw_{n,1}K_{1},
  \end{align*}
  where the last equality follows by telescoping.
\end{proof}

\section{A few examples}\label{sec:examples}
We will now show how to apply Theorem \ref{thm:extendedToBVP} to a few examples of
Fredholm determinants which arise in describing objects of interest in random matrix
theory, growth processes, particle systems, tilings and representation theory. Our
examples include extended determinantal point processes such as the stationary (GUE) Dyson
Brownian motion, the Airy$_2$ process, and the Pearcey process; all of which are limits of
ensembles of non-intersecting directed paths on weights graphs. We also include an
extended determinant point process given by Markov processes on partitions related to the
$z$-measures; this ensemble is {\it not} a limit of a graph-based ensemble of
non-intersecting directed paths. We also show how the identity applies to signed extended
determinantal point processes such as the Airy$_1$ and Airy$_{2\to 1}$ processes.

The proofs of the results in this section are postponed to the appendix.






\subsection{Stationary (GUE) Dyson Brownian motion}\label{sec:gue}

Consider the eigenvalues of an $N\!\times\!N$ Hermitian matrix with each (algebraically
independent) entry diffusing according to a stationary Ornstein-Uhlenbeck process (real valued on the diagonal and complex valued off the diagonal). The eigenvalues of this process are real valued and themselves form a Markov process, called the stationary Dyson Brownian motion. Its stationary marginal distribution is the $N\times N$ Gaussian Unitary Ensemble (GUE) eigenvalue
distribution. We will consider this Dyson Brownian motion process in stationarity and write the $i^{\rm{th}}$ largest eigenvalue at
time $t$ as $\lambda_N(i;t)$. The collection of eigenvalues at time $t$ is written $\lambda_{N}(\cdot;t) = (\lambda_N(1;t),\dots,\lambda_N(N;t))$ and the curve traced out by the $i^{\rm{th}}$ eigenvalue over time is written as $\lambda_N(i;\cdot)$. Then the graphs of $\lambda_N(1;\cdot),\ldots, \lambda_N(N;\cdot)$ form an ensemble of non-intersecting
curves (see, for example, Section 4.3.1 of \cite{AGZ}). This ensemble of curves is indexed
by time $t$ and curve label $i$ and hence can be thought of as a random variable taking
values in the space of continuous curves from $\{1,\ldots,N\}\times \R$ to $\R$. We will
write $\EE$ as the expectation operator for this random variable.

\begin{definition}\label{qdef}
For times $t_1<t_2<\cdots<t_n$ consider functions $q_{t_i}:\R\to\rr$
and let $\bar{q}_{t_i}(x)= 1-q_{t_i}(x)$. For a curve $g:\R\to
\R$ define the functional $\bar{q}$ by $\bar{q}(g) = \prod_{i=1}^{n}
\bar{q}_{t_i}(g(t_i))$. One likewise defines the functional $q(g)=\prod_{i=1}^{n}
q_{t_i}(g(t_i))$.
\end{definition}

The stationary (GUE) Dyson Brownian motion is an extended determinantal point process. In particular this means that for any functions $q_{t_i}$ (as above),
\begin{equation}\label{GUEdpp}
  \EE\!\left[ \prod_{j=1}^{N} \bar{q}(\lambda_N(j;\cdot))\right] = \det\!\big(I-Q\KHrmext\big)_{L^2(\{t_1,\dots,t_n\}\times\rr)}
\end{equation}
as long as both sides are well-defined, where $Q$ is defined as in \eqref{eq:defQ} and
$\KHrmext$ is the \emph{extended Hermite kernel} (see e.g. \cite{tracyWidom-Diff}):
\[\KHrmext(s,x;t,y)=
\begin{dcases*}
  \sum_{k=0}^{N-1}e^{k(s-t)}\varphi_k(x)\varphi_k(y) & if $s\geq t$,\\
  -\sum_{k=N}^\infty e^{k(s-t)}\varphi_k(x)\varphi_k(y) & if $s<t$.
\end{dcases*}\]
Here $\varphi_k(x)=e^{-x^2/2}p_k(x)$ and $p_k$ is the $k$-th normalized Hermite polynomial
(so that $\|\varphi_k\|_2=1$).

Writing \[D=-\tfrac12\big(\Delta-x^2+1\big),\] the \emph{harmonic
  oscillator functions} $\varphi_k$ satisfy $D\varphi_k=k\varphi_k$. Then the
\emph{Hermite kernel}
\[\KHrm(x,y)=\KHrm(0,x;0,y)=\sum_{k=0}^{N-1}\varphi_k(x)\varphi_k(y)\]
acts as the projection operator onto span$\{\varphi_0,\dotsc,\varphi_{N-1}\}$. In the notation of Theorem
\ref{thm:extendedToBVP} we are taking $X=\rr$, $\mu$ the Lebesgue measure, and for $1\leq
i<j\leq n$
\begin{gather*}
  \cw_{t_i,t_j}(x,y)=e^{-(t_j-t_i)D}(x,y)=\sum_{k=0}^\infty e^{-(t_j-t_i)k}\varphi_k(x)\varphi_k(y),\qquad
  K_{t_i}=\KHrm,\\
  \cw_{t_j,t_i}K_{t_i}(x,y)=e^{(t_j-t_i)D}\KHrm(x,y)=\sum_{k=0}^{N-1}e^{(t_j-t_i)k}\varphi_k(x)\varphi_k(y).
\end{gather*}
Applying Theorem \ref{thm:extendedToBVP} we conclude:

\begin{corollary}\label{cor:gue}
  Fix $t_1<\dots<t_n$ and write $\tau=\min_{i=1,\dotsc,n-1}|t_{i+1}-t_i|$. For each $1\leq
  i\leq n$ choose a function $q_{t_i}\in L^1_{\rm loc}(\R)$ satisfying
  $\sup_{x\in\rr}e^{-\kappa x^2}|q_{t_i}(x)|<\infty$ for some
  $\kappa\in(0,\frac1{2\sqrt{2}}\tanh(\tau/\sqrt{2}))$. Then
\begin{equation*}
\EE\!\left[ \prod_{j=1}^{N} q(\lambda^N_j)\right]
  =\det\!\big(I-\KHrm+ Q_{t_1}e^{(t_1-t_2)D} Q_{t_2}\dotsm Q_{t_n}
  e^{(t_n-t_1)D}\KHrm\big)_{L^2(\rr)}.
\end{equation*}
\end{corollary}
Note that we have removed the bars over $q$ and $Q$ by replacing $q_{t_i}$'s by $(1-q_{t_i})$'s.

\subsubsection{Continuum statistics}\label{GUEContStat}

We may now take a continuous time limit of the above formula (in the style of
\cite{CQR}). Consider a function $h:\R\times \R \to [0,\infty]$ and $\ell<r$. Define an
operator $\Gamma^{h}_{\ell,r}$ acting on $L^2(\R)$ as follows: $\Gamma^{h}_{\ell,r}
f(\cdot) = u(r,\cdot)$, where $u(r,\cdot)$ is the solution at time $r$ of
\begin{equation}
\partial_t u =- Du - hu\label{eq:PDEGUE}
\end{equation}
with initial data $u(\ell,x)=f(x)$. By the Feynman-Kac formula we may also express
the action of this operator in terms of a path-integral through a potential $h$ as
\begin{equation}\label{eq:FK}
\Gamma^{h}_{\ell,r}f(x) =
\EE_{b(\ell)=x}\!\left[f(b(r))e^{-\frac12\int_\ell^{r}(2h(s,b(s))+b(s)^2-1)\hs ds}\right]
\end{equation}
where the expectation is over a (standard) Brownian motion $b(\cdot)$ started
at time $\ell$ with $b(\ell)=x$ and run until time $r$. 

Let $t_1=\ell$, $t_n=r$ and the $t_i$ be spaced equally in
between with step size $\delta =(r-\ell)/(n-1)$. Then letting $q_{t_i}(x) =1- \delta h(t_i,x)$
and taking $n\to \infty$ the above formula yields:

\begin{proposition}\label{prop:contGUE}
For any interval $[\ell,r]$ and continuous bounded function $h:\R\times \R \to [0,\infty]$
\begin{equation}\label{continuumstatGUEN}
  \EE \!\left[\prod_{j=1}^{N}\exp\!\left( -\int_{\ell}^{r}\!h(t,\lambda_{N}(j;t))\,dt\right)\right]
  =\det\!\big(I-\KHrm+ \Gamma^{h}_{\ell,r}
  e^{(t_n-t_1)D}\KHrm\big)_{L^2(\rr)}.
\end{equation}
\end{proposition}

\begin{remark}
  The condition on $h$ is not optimal, but it makes the arguments simpler. A different
  class of functions $h$ for which the result holds is the following. Fix a function $g\in
  H^1([\ell,r])$ and set $h(t,x)=0$ for $x<g(t)$ and infinity otherwise. Then the
  left-hand side of (\ref{continuumstatGUEN}) becomes $\PP\!\left[\bigcap_{j=1}^{N}
    \{\lambda_{N}(j;t)<g(t)\, \forall t\in [\ell,r]\}\right]$ and the right-hand side
  makes perfect sense as well, with $\Gamma^h_{\ell,r}$ now being the solution operator of
  a certain boundary operator involving $g$. This case corresponds to calculating the
  probability that on the entire interval $[\ell,r]$, the top curve of the Dyson Brownian
  motion remains below the function $g(t)$. This is the same type of result shown in
  \cite{CQR} for the Airy$_2$ process, and the proof for this case can be easily adapted
  from the arguments in that paper. 
\end{remark}

\subsubsection{Rescaled process}\label{rescaledprocessSec}

Now introduce the rescaled process
\[\widetilde\lambda_N(i;t)=\sqrt{2}N^{1/6}\big(\lambda_N(i;N^{-1/3}t)-\sqrt{2N}\big).\]
Changing variables $x\mapsto \frac1{\sqrt{2}N^{1/6}}x+\sqrt{2N}$, $y\mapsto
\frac1{\sqrt{2}N^{1/6}}y+\sqrt{2N}$ in the kernel accordingly, we immediately obtain:

\begin{corollary}\label{thm:BVPforGUE-OU}
For any $t_1<\dots<t_n$ and functions $q_{t_i}:\R \to \R$, $1\leq i\leq n$, satisfying the same conditions as in
Corollary \ref{cor:gue}, we have
\begin{equation}\label{eqn:BVPforGUE-OU}
\EE\!\left[ \prod_{j=1}^{N} q(\widetilde\lambda_N(j;\cdot))\right] =\det\!\big(I-\widetildeKHrm+Q_{t_1}e^{(t_1-t_2)H_N} Q_{t_2}\dotsm Q_{t_n}
  e^{(t_n-t_1)H_N} \widetildeKHrm\big)_{L^2(\rr)}.
\end{equation}
where the kernel of $\widetildeKHrm$ is given by
\[\widetildeKHrm(x,y)=\frac1{\sqrt{2}N^{1/6}}\KHrm\!\left(\frac{x}{\sqrt{2}N^{1/6}}+\sqrt{2N},
\frac{y}{\sqrt{2}N^{1/6}}+\sqrt{2N}\right)\]
and the operator
\[H_N=-\Delta+x+\frac{x^2}{2N^{2/3}}.\]
\end{corollary}

The above rescaling corresponds to focusing in on the top curves of the Dyson Brownian
motion. In the limit $N$ goes to infinity, $\widetildeKHrm$ converges to the Airy$_2$
kernel $\KAirytwo$ and $H_N$ converges to the Airy Hamiltonian $H$ (defined in the
Introduction and below in Section \ref{Airy2Sec}).  So in the limit as $N$ goes to
infinity we recover the formula for the Airy$_2$ process as expected. The operator in the
Fredholm determinant in the right-hand side of (\ref{eqn:BVPforGUE-OU}) converges in trace
class to the corresponding one with $\KAirytwo$ and $H$, which means that all of the
left-hand side probabilities have limits. This can certainly be proved under some
additional (though not optimal) assumptions on the $q_{t_i}$ as in Corollary
\ref{Airytwobddform}, but we choose to treat the Airy$_2$ process independently.



\subsection{The Airy$_2$ line ensemble}\label{Airy2Sec}

The multi-layer Airy$_2$ process \cite{PS, PNGJ} is the limit of the stationary (GUE)
Dyson Brownian motion under the scaling of Section \ref{rescaledprocessSec}. In
particular for $t\in \R$ consider the point process corresponding to $\left\{\widetilde\lambda_N(i;t): 1\leq i\leq N\right\}$. As $N$ goes to infinity, this point process converges in the vague topology to a limiting point process with an infinite number of simple points which we write as $\left\{\aip(i;t):i\in \Z_{\geq 1}\right\}$ (labeled so that $\aip(i;t)>\aip(j;t)$ for $i<j$). This convergence can be strengthened so that for any fixed set $t_1<t_2<\cdots<t_n$, the $n$-tuple of $\widetilde{\lambda}$-point processes has a limit $\left\{\aip(i;t):i\in \Z_{\geq 1}, t\in \{t_1,\ldots, t_n\}\right\}$. This limiting
collection of point processes is consistent and can be completed to a point process valued
stochastic process indexed by $t\in \R$. This process is called the multi-layer Airy$_2$
process. As it is the limit of a stationary (in $t$) process, it is also stationary.

There exists a continuous version of this process \cite{CH} so that $\aip$ can be thought
of as a random variable taking values in the space of $\Z_{\geq 1}$ indexed, continuous and
non-intersecting curves from $\R$ to $\R$. The convention is that $\aip(1;\cdot)$
represents the top curve (i.e., the limit of $\widetilde\lambda_N(1;\cdot)$). The
continuous version of the multi-layer Airy$_2$ process is called the Airy$_2$ line ensemble.

Since the Dyson Brownian motion was an extended determinantal point process (\ref{GUEdpp}),
so too is the multi-layer Airy$_2$ process. Analogous to (\ref{GUEdpp}), and with the
functional $\bar{q}$ given in Definition \ref{qdef} and operator $Q$ given in (\ref{eq:defQ}),
\begin{equation}\label{Airydpp}
\EE\!\left[ \prod_{j=1}^{\infty} \bar{q}(\aip(j;\cdot))\right] = \det\!\big(I-Q\KAirytwoext\big)_{L^2(\{t_1,\dots,t_n\}\times\rr)}
\end{equation}
where $\KAirytwoext$ is the \emph{extended Airy$_2$ kernel}
\begin{equation}
\KAirytwoext(s,x;t,y)=
\begin{dcases*}
\int_0^{\infty} d\lambda\,e^{-\lambda(s-t)} \Ai(x+\lambda)\Ai(y+\lambda) & if $s\geq t$,\\
-\int_{-\infty}^{0} d\lambda\,e^{-\lambda(s-t)} \Ai(x+\lambda)\Ai(y+\lambda) & if $s<t$,
\end{dcases*}\label{eq:KAi2}
\end{equation}
and $\Ai(\cdot)$ is the Airy function. In order for the above expectation to make sense,
one has to impose conditions on the functions $q_{t_i}$, such as in Corollary
\ref{Airytwobddform}.

To put this example in the setting of Theorem \ref{thm:extendedToBVP} we take $X=\rr$,
$\mu$ the Lebesgue measure, and consider the \emph{Airy Hamiltonian} defined
as \[H=-\Delta +x.\] $H$ has the shifted Airy functions $\Ai_\lambda(x)=\Ai(x-\lambda)$
as its generalized eigenfunctions: $H\!\Ai_\lambda(x) = \lambda\!\Ai_\lambda(x)$. Define
the Airy$_2$ kernel $\KAirytwo$ as the projection of $H$ onto its negative generalized
eigenspace:
\[\KAirytwo(x,y)=\int_0^{\infty} d\lambda \Ai(x+\lambda)\Ai(y+\lambda).\]
Then it is not hard to check that, in the notation of Theorem \ref{thm:extendedToBVP},
\eqref{eq:KAi2} corresponds to taking, for $1\leq i<j\leq n$,
\begin{gather*}
  \cw_{t_i,t_j}(x,y)=e^{-(t_j-t_i)H}(x,y)=\int_{-\infty}^{\infty} d\lambda\,e^{\lambda(t_j-t_i)} \Ai(x+\lambda)\Ai(y+\lambda),\qquad
  K_{t_i}=\KAirytwo,\\
  \cw_{t_j,t_i}K_{t_i}(x,y)=e^{(t_j-t_i)H}\KAirytwo(x,y)=\int_0^\infty d\lambda\,e^{-\lambda(t_j-t_i)}\Ai(x+\lambda)\Ai(y+\lambda).
\end{gather*}
Note that $\cw_{t_i,t_j}$ is only well-defined on the range of $\KAirytwo$. Applying
Theorem \ref{thm:extendedToBVP} allows to conclude:

\begin{corollary}\label{Airytwobddform}
  Fix $t_1<\dots<t_n$ and let $\tau=\min_{i=1,\dotsc,n-1}\{|t_{i+1}-t_i|\}$. Choose
  functions $q_{t_i}\in L^1_{\rm loc}(\R)$, $1\leq i\leq n$, such that
  $\sup_{x\geq0}e^{-rx}|q_{t_i}(x)|<\infty$ for some $0<r<\tau$ and
  $\sup_{x<0}\varphi(x)|1-q_{t_i}(x)|<\infty$ for some function $\varphi(x)$ such that
  $\int_{-\infty}^0 dx\,e^{-2(t_n-t_1) x}\varphi(x)^{-2}<\infty$, $1\leq i\leq n$.
  Then
  \begin{equation*}
    \EE\!\left[ \prod_{j=1}^{\infty} q(\aip(j;\cdot))\right]
    =\det\!\big(I-\KAirytwo+ Q_{t_1}e^{(t_1-t_2)H} Q_{t_2}\dotsm Q_{t_n}
    e^{(t_n-t_1)H}\KAirytwo\big)_{L^2(\rr)}.
  \end{equation*}
\end{corollary}

This formula is also the limit of the right-hand side of (\ref{eqn:BVPforGUE-OU}) as $N$
goes to infinity.

Since the Airy line ensemble is a continuous version of the multi-layer Airy$_2$ process,
we may take a continuum limit of the above formula, in the same manner as done in Section
\ref{GUEContStat}. The PDE which $\Gamma^{h}_{\ell,r}$ is solving is now $\partial_t u =
-Hu - hu$ (corresponding to replacing $D$ by $H$ in \eqref{eq:PDEGUE}) and the result is
that for any interval $[\ell,r]$ and suitable function $h:\R\times \R \to [0,\infty]$
(for example $h$ can be taken to be bounded, continuous, and such that $h(t,x)=0$ for any
$t\in[\ell,r]$ and $x<M$ for some $M\in\rr$),
\begin{equation}\label{continuumstatAiry}
\EE\!\left[\prod_{j=1}^{\infty}\exp\!\left(-\int_{\ell}^{r}h(t,\aip(j;t)) dt\right) \right]
  =\det\!\big(I-\KAirytwo+ \Gamma^{h}_{\ell,r}
  e^{(r-\ell)H}\KAirytwo\big)_{L^2(\rr)}.
\end{equation}
We will omit the proof of this statement, which can be adapted from the proofs of Proposition
\ref{prop:contGUE} and Corollary \ref{Airytwobddform} together with the proof of
Proposition 3.2 in \cite{CQR}. Taking $h(t,x)$ to be 0 for $x<g(t)$ and infinity otherwise
we recover Theorem 2 of \cite{CQR}.

\subsection{The Pearcey process}\label{sec:pearcey}

There are many other multi-layer processes which arise as scaling limits of
non-intersecting ensembles of Brownian motions (or similar diffusions) for which we can
apply Theorem \ref{thm:extendedToBVP} (see for instance Airy-like processes
\cite{ADvM,AFvM,bbp,BorPeche}; bulk limits such as the Sine process \cite{tracyWidom-Diff}, Pearcey process
\cite{tracyWidom-Pearcey} or Tacnode process \cite{BDTac}; hard edge limits like the
Bessel process \cite{KTBes}).

To illustrate this point we will show how a Fredholm determinant involving the Pearcey
kernel can be rewritten via Theorem \ref{thm:extendedToBVP}.

Let us briefly and informally recall one way the Pearcey process arises as a scaling limit
of Brownian bridges. Consider $2N$ Brownian bridges on the time interval $[-N,N]$ such
that all $2N$ of them start at height 0 and $N$ of them end at height $b$ and the other
$N$ end at height $-b$. Condition these Brownian bridges not to intersect (as can be done
by spacing their starting and ending points by $\varepsilon$ and letting $\varepsilon$ go
to zero). When $b=0$ the limit shape of the ensemble of conditioned Brownian bridges has a
limit shape which is elliptical (and the ensemble is sometimes called a watermelon) and
the fluctuations around the top of this limit shape are described (in the limit as $N$
goes to infinity) by the Airy$_2$ line ensemble minus a parabolic shift.

When the endpoints parameter $b=cN$, the limit shape has a cusp at some time $t=c'N$,
where $c'\in(-1,1)$ is a function of $c$. For $t_1<t_2<\cdots<t_n$, the $N$-tuple of point
processes formed by the heights (properly centered and normalized by $N^{1/4}$ near the
height of the cusp) of the Brownian bridges at times $c'N+t_iN^{1/2}$, $1\leq i \leq n$,
converges in the vague topology as $N$ goes to infinity to a limit which is called the
\emph{Pearcey process}, $\prc$, see \cite{ABK,brezinHikami,brezinHikami2,OkResh,tracyWidom-Pearcey}. It is a
point process valued stochastic process indexed by $t\in \R$. At each time $t$ the point
process can be indexed by $\Z$ as $\left\{\prc(j;t):j\in \Z\right\}$.

Analogously to (\ref{GUEdpp}), and with the functional $\bar{q}$ given in Definition \ref{qdef} and operator $Q$ given in (\ref{eq:defQ}),
\begin{equation}\label{Pearceydpp}
\EE\!\left[ \prod_{j=-\infty}^{\infty} \bar{q}(\prc(j;\cdot)\right] = \det\!\big(I-Q\KPrcext\big)_{L^2(\{t_1,\dots,t_n\}\times\rr)},
\end{equation}
where $\KPrcext$ is the \emph{extended Pearcey kernel}
\begin{multline}
\KPrcext(s,x;t,y)=-\frac{1}{\sqrt{4\pi(t-s)}}\exp\left(-\frac{(y-x)^2}{4(t-s)}\right)\uno{t>s}\\
+\frac{1}{(2\pi \I)^2} \int_{C}du \int_{-\I\infty}^{\I\infty}dv
\frac{e^{-v^4/4+t v^2/2-yv}}{e^{-u^4/4+s u^2/2-xu}}\frac1{v-u},
\end{multline}
and where $C$ is the contour consisting of the rays going from $\pm\infty e^{\I\pi/4}$ to 0
and from 0 to $\pm\infty e^{-\I\pi/4}$.

In the setting of Theorem  \ref{thm:extendedToBVP} we take $X=\rr$,
$\mu$ the Lebesgue measure, and for $t_i<t_j$ define
\begin{gather*}
  \cw_{t_i,t_j}=e^{\frac12(t_j-t_i)\Delta},\qquad
  K_{t_i}(x,y)=\KPrc^{t_i}(x,y):=\KPrcext(t_i,x;t_i,y),\\
  \cw_{t_j,t_i}K_{t_i}(x,y)=\KPrcext(t_j,x;t_i,y).
\end{gather*}
The semigroup property is obviously satisfied, while for $i<j$
\begin{multline*}
  \cw_{t_i,t_j}\KPrc^{t_j}(x,y)=\int_{-\infty}^\infty
  dz\,\frac1{\sqrt{2\pi(t_j-t_i)}}e^{-\frac{(x-z)^2}{2(t_j-t_i)}}\frac{1}{(2\pi \I)^2}
  \int_{C}du \int_{-\I\infty}^{\I\infty}dv
  \frac{e^{-\frac{v^4}4+t_j\frac{v^2}2-yv}}{e^{-\frac{u^4}4+t_j \frac{u^2}2-zu}}\frac1{v-u}\\
  =\frac{1}{(2\pi \I)^2} \int_{C}du \int_{-\I\infty}^{\I\infty}dv\,
  \frac{e^{-\frac{v^4}4+t_j\frac{v^2}2-yv}}{e^{-\frac{u^4}4+t_i\frac{
        u^2}2-xu}}\frac1{v-u}=\KPrcext(t_i,x;t_j,y)+\cw_{t_i,t_j}=\KPrc^{t_i}\cw_{t_i,t_j}(x,y),
\end{multline*}
where the second equality follows from computing a simple Gaussian integral and the last
equality is obtained similarly. Likewise one can check that for $i<j$ we have
$\cw_{t_i,t_j}\cw_{t_j,t_i}\KPrc^{t_i}=\KPrc^{t_i}$. Hence Assumption \ref{assum:1} is
satisfied, and from Theorem \ref{thm:extendedToBVP} we deduce the following:

\begin{corollary}\label{thm:pearcey}
  For any $t_1<t_2<\dots<t_n$ and functions $q_{t_i}:\R\to\R$, $1\leq i\leq n$ so
  that Assumptions \ref{assum:0} and \ref{assum:2} are satisfied, we have
  \begin{equation}
    \EE\!\left[ \prod_{j=-\infty}^{\infty} q(\prc(j;\cdot))\right]
    =\det\!\big(I-\KPrc^{t_1}+Q_{t_1}e^{\frac12(t_2-t_1)\Delta}Q_{t_2}\cdots e^{\frac12(t_n-t_{n-1})\Delta}Q_{t_n} e^{\frac12(t_1-t_n)\Delta} \KPrc^{t_1}\big)_{L^2(\rr)}.
  \end{equation}
  In particular, the formula holds for the case $q_{t_i}(x)=\uno{x\leq a_i}$.
\end{corollary}

We do not attempt here to provide more general conditions on the functions $q_{t_i}$ so that
the formula holds.



\subsection{The Airy$_1$ and Airy\texorpdfstring{$_{2\to 1}$}{2->1} processes}\label{Airy12Sec}

All of the examples considered thus far have involved probability measures on ensembles of
non-intersecting paths or their scaling limits. Going back to the discrete setting of
Theorem \ref{combthm}, there was no condition that the measure on non-intersecting paths be
positive. This condition is not met, for example, in the case of the Airy$_1$ and
Airy$_{2\to 1}$ processes. These are real valued stochastic processes which are the
scaling limits of marginals of measures (not entirely positive) on non-intersecting paths
\cite{sasamotoAiry1,bfps,bfp,bfs}. Even though the ensemble measure is not
entirely positive, the marginal is a probability measure.

We will focus on the Airy$_{2\to 1}$ process obtained in \cite{bfs}, since a similar result to that which we now
state has already shown up in \cite{quastelRemAiry1}. The Airy$_{2\to 1}$ process is a
continuous time (non-stationary) real valued process $\Bt:\R\to \R$ given by its
finite-dimensional distributions
\begin{equation}\label{eqTransAiryProcess}
\pp\!\left(\bigcap_{k=1}^n\{\Bt(t_k)\le x_k\}\right)=  \det\!\big(I-\chi K_{2\to1}\big)_{L^2(\{t_1,\dots,t_n\}\times\mathbb{R})}
\end{equation}
for $t_1<\dots<t_n$, where $\chi f(t_i,x)=\uno{x\geq x_i}f(x)$ and
\begin{multline}\label{eqtildeKCompleteInfinity}
\Ktwooneext(s,x;t,y)=-\frac{1}{\sqrt{4\pi(t-s)}}\exp\left(-\frac{(\tilde y-\tilde x)^2}{4(t-s)}\right)\uno{t>s}\\
+\frac{1}{(2\pi \I)^2} \int_{\gamma_+}dw \int_{\gamma_-}dz\,
\frac{e^{w^3/3+t w^2-\tilde y w}}{e^{z^3/3+s z^2-\tilde x z}} \frac{2w}{(z-w)(z+w)}
\end{multline}
with
\[\tilde x=x-(s^-)^2,\qquad\tilde y=y-(t^-)^2,\]
notation $r^-=\min\{0,r\}$, and the paths $\gamma_+,\gamma_-$ satisfying
$-\gamma_+\subseteq\gamma_-$ with \mbox{$\gamma_+:e^{\I\phi_+}\infty\to
  e^{-\I\phi_+}\infty$}, \mbox{$\gamma_-:e^{-\I\phi_-}\infty\to e^{\I\phi_-}\infty$} for
some $\phi_+\in (\pi/3,\pi/2)$, $\phi_-\in (\pi/2,\pi-\phi_+)$. The Airy$_{2\to 1}$
process crosses over between the Airy$_2$ and the Airy$_1$ processes in the sense that
$\Bt(t+\tau)$ converges to $2^{1/3}\aipo(2^{-2/3}t)$ as $\tau\to\infty$ and to $\aip(1;t)$
(the Airy$_2$ process, i.e. the top line of the multi-layer Airy$_2$ process) when
$\tau\to-\infty$ (in the sense of finite dimensional distributions). It is expected to
govern the asymptotic spatial fluctuations in random growth models when the initial
conditions are deterministic near the point where the hydrodynamic profile changes from
flat to curved. In particular, it is shown in \cite{bfs} that it governs the asymptotic
fluctuations near the profile switch point for the totally asymmetric simple exclusion process starting with particles
only at the even negative integers.

We take again $X=\rr$ and $\mu$ the Lebesgue measure, and for $i<j$ we define
\begin{gather*}
  \cw_{t_i,t_j}(x,y)=e^{(t_j-t_i)\Delta}(x-(t_i^-)^2,y-(t_j^-)^2),\quad\,\,
  K_{t_i}(x,y)=\Ktwoone^{t_i}(x,y):=\Ktwooneext(t_i,x;t_i,y)\\
  \cw_{t_j,t_i}K_{t_i}(x,y)=\Ktwooneext(t_j,x;t_i,y).\
\end{gather*}
Proceeding as in Section \ref{sec:pearcey} one checks that these choices satisfy
Assumption \ref{assum:1}, and hence (under the additional assumptions) we may
apply Theorem \ref{thm:extendedToBVP}. Using the translation invariance of the heat kernel
to rearrange the shifts appearing in the resulting formula we get:

\begin{corollary}\label{cor:2to1}
For any $t_1<t_2<\cdots<t_n$, we have
\begin{multline}\label{eq:2to1}
  \pp\!\left(\bigcap_{k=1}^n\{\Bt(t_k)\le x_k\}\right)\\
  =\det\!\big(I-\Ktwoone^{t_1}+\bar P_{\tilde x_1}e^{(t_2-t_1)\Delta}\bar P_{\tilde
    x_2}\cdots e^{(t_n-t_{n-1})\Delta}\bar P_{\tilde x_n} e^{(t_1-t_n)\Delta}
  \Ktwoone^{t_1}\big)_{L^2(\rr)},
\end{multline}
where $\tilde x_i=x_i-(t_i^-)^2$, $\bar{P}_{a}f(x) = \uno{x\leq a} f(x)$ and $\Ktwoone^{t_1}(x,y)=\Ktwooneext(t_1,x+(t_1^-)^2;t_1,y+(t_1^-)^2)$.
\end{corollary}

In the formula, $e^{(t_1-t_n)\Delta}K_{2\to1}^{t_1}$ should be interpreted as
$\Ktwooneext(t_n,x+(t_n^-)^2;t_1,y+(t_1^-)^2)$. One can use this formula directly to
recover the analogous path-integral kernel formulas for the Airy$_1$ and Airy$_2$
processes in the appropriate limits, and thus show that $\Bt$ interpolates between
these two processes.

\subsection{Markov processes on partitions and $z$-measures}

The $z$-measures are a remarkable family of probability distributions on
partitions that arise in representation theory of the infinite-symmetric group.
They can be viewed as determinantal point processes on the one-dimensional
lattice with infinite many particles, and they degenerate to a variety of
well-known discrete and continuous determinantal point processes, see
\cite{BO-MRL,Ol-Springer,BO-Gamma} and references therein.

In \cite{BOMarkov}, a Markov process on partitions that preserves the
$z$-measures was constructed. Its dynamical correlation functions are
determinantal, and they can be described via the corresponding extended kernel,
see Section 6 of \cite{BOMarkov}. One particular limit of this Markov process
can be seen as `space-like' space-time sections of the multilayer polynuclear
growth process of \cite{PS}, see \cite{BO-Planch}.

Note that it is not known how to obtain the $z$-measures and the corresponding
Markov processes as a limit of an ensemble of nonintersecting paths. However,
these objects can be viewed as an \emph{anaytic continuation} of an
ensemble of nonintersecting birth-and-death processes \emph{in the number of
paths}, see Section 6.5 of \cite{BOMarkov}.

By
encoding a partition $\lambda=(\lambda_1\geq \lambda_2\ge \cdots)$ by the point
configuration $\{\lambda_i - i +\tfrac{1}{2}\}_{i\ge 1}$, the Markov
process can be written as $\left\{\Zproc(j;t): j\in \Z_{\geq 1}, t\in \R\right\}$. Here
$\Zproc(j;t)$ takes values in $\Z'= \Z+\tfrac{1}{2}$.

Similarly to (\ref{GUEdpp}), and with the functional $\bar{q}$ given in Definition \ref{qdef} and operator $Q$ given in (\ref{eq:defQ}),
\begin{equation}\label{Zdpp}
\EE\!\left[ \prod_{j=1}^{\infty} \bar{q}(\Zproc(j;\cdot))\right] =
\det\!\big(I-Q\KZext\big)_{L^2(\{t_1,\dots,t_n\}\times\zz')}
\end{equation}
where $\KZext$ is the \emph{extended hypergeometric kernel} which we will now
define.

For parameters $z,z'\in\mathbb{C}$ such that either
$z'=\bar{z}\in\mathbb{C}\setminus\mathbb{Z}$ or $m<z,z'<m+1$ for a
$m\in\mathbb{Z}$, and $\xi\in (0,1)$ define a second order difference
operator $\DZ$ on $\Z'$, depending on $(z,z',\xi)$ and acting on
functions $f(\cdot)\in \ell^2(\Z')$ as follows
\begin{multline}
(\DZ f)(x) = \sqrt{\xi(z+x+\tfrac{1}{2})(z'+x+\tfrac{1}{2})} f(x+1)\\
+  \sqrt{\xi(z+x-\tfrac{1}{2})(z'+x-\tfrac{1}{2})} f(x-1) - (x+\xi(z+z'+x)) f(x).
\end{multline}
This is a self-adjoint operator with discrete simple spectrum $(1-\xi)\Z'$. Its eigenfunctions $\psi_a$,
\begin{equation*}
\DZ\psi_a = (1-\xi)a \, \psi_a,
\end{equation*}
are explicitly written through the Gauss hypergeometric function (see \cite{BOMarkov},
equation (5.1)). We normalize them by the condition $\|\psi_a\|_{\ell^2(\Z')}
=1$. Then
\begin{equation}\label{eq:extKernelZ}
\KZext(s,x;t,y)=
    \begin{dcases*}
      \sum_{a\in \Z'_{+}} e^{-a(s-t)} \psi_{a}(x)\psi_{a}(y) & if $s\geq t$,\\
      -\sum_{a\in \Z'_{-}} e^{-a(s-t)} \psi_{a}(x)\psi_{a}(y) & if $s<t$,
    \end{dcases*}
\end{equation}
where $\Z'_{\pm} = \{\pm \tfrac{1}{2},\pm \tfrac{3}{2},\pm\tfrac{5}{2},\ldots\}$.

Let $\DZ' = -(1-\xi)^{-1} \DZ$. Then in the setting of Theorem \ref{thm:extendedToBVP},
\eqref{eq:extKernelZ} corresponds to taking $X=\Z'$, $\mu$ the counting measure and, for
$i<j$,
\begin{gather*}
  \cw_{t_i,t_j}=e^{(t_j-t_i)\DZ'},\qquad
  K_{t_i}(x,y)=\KZ(x,y):=\KZext(0,x;0,y)=\sum_{a\in \Z'_{+}}\psi_{a}(x)\psi_{a}(y),\\
  \cw_{t_j,t_i}K_{t_i}(x,y)=\KZext(t_j,x;t_i,y)=\sum_{a\in \Z'_{+}} e^{-a(t_j-t_i)} \psi_{a}(x)\psi_{a}(y).
\end{gather*}
Thus from Theorem \ref{thm:extendedToBVP} we deduce the following:

\begin{corollary}\label{thm:Zmeas}
For any $t_1<t_2<\dots<t_n$ and functions $q_{t_i}:\mathbb{Z}'\to\mathbb{R}$,
$1\le i\le n$, so that Assumptions 1 and 3 are satisfied, we have
\begin{multline}
\EE\!\left[ \prod_{j=1}^{\infty} q(\Zproc(j;\cdot))\right]\\
=\det\!\big(I-\KZ+Q_{t_1}e^{(t_2-t_1)\DZ'}Q_{t_2}\cdots e^{(t_n-t_{n-1})\DZ'}Q_{t_n} e^{(t_1-t_n)\DZ'} \KZ\big)_{L^2(\zz')}.
\end{multline}
\end{corollary}

Let us remark that if all the functions $q_{t_i}$ have finite support then all
the needed analytic assumptions are automatically satisfied because we are
working in an $L^2$ space on a finite set. Of course, such a restriction is
unnecessarily harsh, but we will not pursue this issue here any further.

\appendix

\section{Proofs of the results from Section \ref{sec:examples}}
\label{sec:app}

We recall the following facts about trace class and Hilbert-Schmidt norms (see e.g.
\cite{Simon}) of operators in $L^2(X)$ for some measurable space $(X,\Sigma,\mu)$, which we will use repeatedly without reference:
\begin{gather*}
  \|AB\|_1\leq\|A\|_2\|B\|_2,\qquad\|AB\|_1\leq\|A\|_{\rm op}\|B\|_1,\\
  \shortintertext{and if $A$ has integral kernel $A(x,y)$,}
  \|A\|_2=\left(\int \mu(dx)\mu(dy)\,|A(x,y)|^2\right)^{1/2},
\end{gather*}
for each $A,B$ in the appropriate space, where $\|\cdot\|_{\rm op}$ denotes the operator
norm in $L^2(X)$.

Throughout this section $c$ and $c'$ will denote positive constants whose value may change
from line to line.

\begin{proof}[Proof of Corollary \ref{cor:gue}]
  Checking Assumption \ref{assum:1} is straightforward. We will take in this case
  $V_{t_i}=V_{t_i}'=U_{t_i}=U_{t_i}'=I$, and thus Assumption \ref{assum:0} is contained in
  Assumption \ref{assum:2}, which we check next.

  Condition (i) is trivial. Given functions $\psi_1$ and $\psi_2$
  write $\psi_1\otimes\psi_2$ for the kernel $\psi_1(x)\psi_2(y)$ and let $\phi$ be any
  function with $\int\!\phi^2=1$. Then we can write
  $\oQ_{t_i}(\varphi_k\otimes\varphi_k)=(\oQ_{t_i}\varphi_k\otimes\phi)(\phi\otimes\varphi_k)$,
  so that
  $\|\oQ_{t_i}(\varphi_k\otimes\varphi_k)\|_1\leq\|\oQ_{t_i}\varphi_k\otimes\phi\|_2\|\phi\otimes\varphi_k\|_2$
  (note that we need to consider the operators with the bars because of the remark
  following the statement of the corollary). Now slightly abusing notation to write
  $\|\cdot\|_2$ both for the Hilbert-Schmidt norm of operators in $L^2(\rr)$ and for the
  norm of this last space, we have
  \[\|\oQ_{t_i}\varphi_k\otimes\phi\|_2=\|(1-q_{t_i})\varphi_k\|_2\|\phi\|_2=\|(1-q_{t_i})\varphi_k\|_2
  \qqand\|\varphi_k\otimes\phi\|_2=\|\varphi_k\|_2\|\phi\|_2=1,\]
  so for $1\leq i,j\leq n$ we have
  \[\|\oQ_{t_i}e^{(t_j-t_i)D}\KHrm\|_1\leq\sum_{k=0}^{N-1}e^{(t_j-t_i)k}\|\oQ_{t_i}(\varphi_k\otimes\varphi_k)\|_1
  \leq\sum_{k=0}^{n-1}e^{(t_j-t_i)k}\|\bar q_{t_i}\varphi_k\|_2<\infty,\] since
  $|\varphi_k(x)|\leq cx^ke^{-x^2/2}$ and $|\bar q_{t_i}(x)|\leq ce^{\kappa x^2}$ where
  $\kappa<\tfrac1{2\sqrt{2}}\tanh(\tau/\sqrt{2})<\frac12$. Hence the only thing left to check in (ii) is that
  $\|\oQ_{t_i}e^{-(t_j-t_i)D}\|_1<\infty$ for $i<j$. To that end we use the Feynman-Kac
  representation to write (setting $t=\tfrac12(t_j-t_i)$)
  \begin{equation}
    e^{-tD}(x,y)=\tfrac1{\sqrt{2\pi
        t}}e^{-(x-y)^2/2t}\ee_{b(0)=x,\,b(t)=y}\!\left[e^{-\frac12\int_{0}^{t}(b(s)^2-1)\hs
      ds}\right],\label{eq:FK-GUE}
  \end{equation}
  where $b(s)$ denotes a standard Brownian motion and the subscript in the expectation
  means that it is conditioned (in the sense of a Brownian bridge) to go from $x$ at time
  0 to $y$ at time $t$. Then
  \begin{align*}
    \|\oQ_{t_i}e^{-tD}\|^2_2&=\int_{\rr^2}dx\,dy\,\bar q_{t_i}(x)^2\frac1{2\pi t}
    e^{-(x-y)^2/t}\ee_{b(0)=x,b(t)=y}\!\left[e^{-\frac12\int_{0}^{t}(b(s)^2-1)\hs ds}\right]^2\\
    &\leq\int_{\rr^2}dx\,dy\, ce^{2\kappa x^2}\frac1{2\pi t}
    e^{-(x-y)^2/2t}\ee_{b(0)=x,b(t)=y}\!\left[e^{-\int_{0}^{t}(b(s)^2-1)\hs ds}\right]\\
    &\leq \int_{-\infty}^\infty dx\,c\frac{e^{2\kappa x^2+t}}{\sqrt{2\pi
        t}}\ee_{b(0)=x}\!\left[e^{-\int_{0}^{t}b(s)^2\hs ds}\right]
    \leq c'\int_{-\infty}^\infty dx\,e^{2\kappa x^2-{\rm tanh}(\sqrt{2}t)x^2/\sqrt{2}}<\infty
  \end{align*}
  by our assumption on $\kappa$, where in the last inequality we have used (1.9.3) in
  \cite{handbookBM}. In the same way we have $\|e^{-tD}\|_2<\infty$ and then
  $\|\oQ_{t_i}e^{-(t_j-t_i)D}\|_1\leq\|\oQ_{t_i}e^{-tD}\|_2\|e^{-tD}\|_2<\infty$.

  Finally, Assumption \ref{assum:2}(iii) follows from Assumption \ref{assum:2}(ii) thanks
  to the observation in Remark \ref{rem:assum2}.
\end{proof}

\begin{proof}[Proof of Proposition \ref{prop:contGUE}]
  Using the notation introduced before the statement of the result, it is clear that the
  functions $q_{t_i}$ satisfy the assumptions appearing in Corollary \ref{cor:gue}, so that
  \begin{equation}\label{eq:prodcontGUE}
    \EE\!\left[ \prod_{j=1}^N\prod_{i=1}^n(1-\delta h(t_j,\lambda^N_j(t_i)))\right]
    =\det\!\big(I-\KHrm+ Q_{t_1}e^{(t_1-t_2)D}\dotsm Q_{t_n}
    e^{(t_n-t_1)D}\KHrm\big)_{L^2(\rr)}.
  \end{equation}
  The left-hand side equals
  \begin{multline}\label{eq:riemannsum}
    \ee\!\left[\prod_{j=1}^N\exp\!\left(\sum_{i=1}^n\big[\log(1-\delta
        h(t_i,\lambda^N_j(t_i)))\right)\right]
    =\ee\!\left[\prod_{j=1}^N\exp\!\left(-\delta\sum_{i=1}^nh(t_i,\lambda^N_j(t_i))+n\mathcal{O}(\delta^2)\right)\right]\\
    \xrightarrow[n\to\infty]{}\ee\!\left[\prod_{j=1}^N\exp\!\left(-\int_\ell^r\,h(t,\lambda^N_j(t))\,dt\right)\right]
  \end{multline}
  by the dominated convergence theorem.

  For the right-hand side of \eqref{eq:prodcontGUE}, writing
  $\Gamma^{h,n}_{\ell,r}=Q_{t_1}e^{(t_1-t_2)D} Q_{t_2}\dotsm e^{(t_{n-1}-t_n)D}Q_{t_n}$
  one can use the Feynman-Kac representation on each interval $[t_i,t_{i+1}]$ as in
  \eqref{eq:FK-GUE} (see also \eqref{eq:FKPs}) to deduce that $\Gamma^{h,n}_{\ell,r}$ has
  kernel
  \[\Gamma^{h,n}_{\ell,r}(x,y)
  =\tfrac1{\sqrt{2\pi(r-\ell)}}e^{-(x-y)^2/2(r-\ell)}\ee_{b(\ell)=x,\,
    b(r)=y}\left[e^{\sum_{i=0}^{n}\log(1-\delta h(t_{i},
      b(t_{i})))-\frac12\int_{\ell}^{r}(b(s)^2-1)\,ds}\right],\] where $b(s)$
  is a Brownian bridge (with diffusion coefficient 2) run from $x$ at time $\ell$ to $y$
  at time $r$. Then using \eqref{eq:FK} we deduce that
  \begin{multline*}
    \left[\Gamma^{h,n}_{\ell,r}-\Gamma^h_{\ell,r}\right]\!(x,y)
    =\tfrac1{\sqrt{2\pi(r-\ell)}}e^{-(x-y)^2/2(r-\ell)}\ee_{b(r)=x,\,b(\ell)=y}\!\left[e^{-\frac12\int_{\ell}^{r}(2h(s,b(s))+b(s)^2-1)\hs
        ds}\right.\\
    \cdot\left.\left(e^{\sum_{i=0}^{n}\log(1-\delta
          h(t_{i},b(t_i)))+\int_{\ell}^{r}h(s,b(s))\hs ds}-1\right)\right]
  \end{multline*}
  for small enough $\delta$. Since $h$ is bounded and continuous, the random variable inside the expectation goes to
  0 almost surely as $n\to\infty$ using a similar argument as in \eqref{eq:riemannsum},
  and thus since this random variable is bounded by $ce^{-\int_{\ell}^r b(s)^2\hs ds}$ the
  whole expected value goes to 0 as $n\to\infty$ by the dominated convergence theorem. If
  we now define the multiplication operator $Mf(x)=\phi(x)f(x)$ with
  $\phi(x)=(1+x^2)^{-1/2}$ then the above argument gives
  $(\Gamma^{h,n}_{\ell,r}-\Gamma^h_{\ell,r})M(x,y)\to0$ as $n\to\infty$ for all $x,y$. To deduce
  that $\|(\Gamma^{h,n}_{\ell,r}-\Gamma^h_{\ell,r})M\|_2\to0$ as $n\to\infty$ we use the
  dominated convergence theorem again together with the fact that
  $(\Gamma^{h,n}_{\ell,r}-\Gamma^h_{\ell,r})M$ satisfies
  \begin{align*}
    &\int_{\rr^2}dx\,dy\,\big[(\Gamma^{h,n}_{\ell,r}-\Gamma^h_{\ell,r})M(x,y)\big]^2 \leq
   c\int_{\rr^2}dx\,dy\,e^{-\frac{(x-y)^2}{r-\ell}}\ee_{b(\ell)=x,\,b(r)=y}\!\left[e^{-\frac12\int_{\ell}^{r}b(s)^2\,ds}\right]^2\phi(y)^2\\
    &\qquad\leq c\left[\int_{-\infty}^\infty
      dy\,\phi(y)^4\right]^{1/2}\int_{-\infty}^\infty dx\left[\int_{-\infty}^\infty
      dy\,e^{-\frac{2(x-y)^2}{r-\ell}}\ee_{b(\ell)=x,\,b(r)=y}\!\left[e^{-2\int_{\ell}^{r}b(s)^2\,ds}\right]\right]^{1/2}\\
    &\qquad\leq c\|\phi\|^2_4\int_{-\infty}^\infty dx\left[\int_{-\infty}^\infty
      dy\,e^{-\frac{(x-y)^2}{2(r-\ell)}}\ee_{b(\ell)=x,\,b(r)=y}\!\left[e^{-2\int_{\ell}^{r}b(s)^2\,ds}\right]\right]^{1/2}\\
    &\qquad\leq c\|\phi\|^2_4\int_{-\infty}^\infty dx\,\ee_{b(\ell)=x}
    \left[e^{-2\int_{\ell}^{r}\,b(s)^2\,ds}\right]^{1/2}
    =c'\|\phi\|^2_4\left[\int_{-\infty}^\infty dx\,e^{-{\rm tanh}(2(r-\ell))x^2}\right]^{1/2}<\infty,
  \end{align*}
  where we have used the Cauchy-Schwartz inequality and the last equality
  follows from (1.9.3) of \cite{handbookBM}. Checking that
  $\|M^{-1}e^{(r-\ell)D}\KHrm\|_2<\infty$ is simple as in the proof of Corollary
  \ref{cor:gue}, so from the above we deduce that
  \begin{multline*}
    \big\|\big(\KHrm-\Gamma^{h,n}_{\ell,r}e^{(r-\ell)D}\KHrm\big)-\big(\KHrm-\Gamma^h_{\ell,r}e^{(r-\ell)D}\KHrm\big)\big\|_1\\
    \leq\big\|\big(\Gamma^h_{\ell,r}-\Gamma^{h,n}_{\ell,r}\big)M\big\|_2\big\|M^{-1}e^{(r-\ell)D}\KHrm\big\|_2
    \xrightarrow[n\to\infty]{}0.
  \end{multline*}
  Since the mapping $A\mapsto\det(I+A)_{L^2(\rr)}$ is continuous in the space of trace
  class operators (see \cite{Simon}), we deduce that the right-hand side of \eqref{eq:prodcontGUE} converges to
  $\det(I-\KHrm+\Gamma^h_{\ell,r}e^{(r-\ell)D}\KHrm)_{L^2(\rr)}$, and hence
  \begin{equation*}
    \EE\!\left[\prod_{j=1}^{N}\exp\!\left( -\int_{\ell}^{r}\!h(t,\lambda^{N}_j(t))\,dt\right)\right]
    =\det\!\big(I-\KHrm+ \Gamma^{h}_{\ell,r}
    e^{(t_n-t_1)D}\KHrm\big)_{L^2(\rr)}.\label{eq:gammaHat}\qedhere
  \end{equation*}

\end{proof}

\begin{proof}[Proof of Corollary \ref{Airytwobddform}]
  Fix $f\in L^2(\rr)$ and write $\phi(x)=e^{rx}\uno{x\geq0}+\varphi(x)^{-1}\uno{x<0}$. Then
  for $i<j$ (note that as in the proof of Corollary \ref{cor:gue} we need to consider the
  operators with bars), writing $\hat f(\lambda)=\int_{-\infty}^\infty
  dx\,\Ai(x+\lambda)f(x)$ we have
  \begin{equation}
  \begin{aligned}
    \|\oQ_{t_i}&e^{(t_i-t_j)H}f\|_2^2\leq c\int_{-\infty}^\infty
    dx\!\left[\int_{\rr^2}dy\,d\lambda\,\phi(x)e^{\lambda(t_i-t_j)}\Ai(x+\lambda)\Ai(y+\lambda)f(y)\right]^2\\
    &\qquad\leq c\int_{-\infty}^\infty dx\!\left[\int_{-\infty}^\infty
      d\lambda\,\phi(x)^2e^{2\lambda(t_j-t_i)}\Ai(x+\lambda)^2\right]
    \left[\int_{-\infty}^\infty d\lambda\,\hat f(\lambda)^2\right]\\
    &\qquad=c\|f\|^2_2\int_{-\infty}^\infty dx\,\phi(x)^2e^{-2(t_j-t_i)x}\int_{-\infty}^\infty
    d\lambda\,e^{2(\lambda+1)(t_j-t_i)}\Ai(\lambda)^2\leq c'\|f\|^2_2.
  \end{aligned}\label{eq:oQetH}
\end{equation}
where in the last line we have used the Parseval identity for the Airy transform
$\int\!\hat f^2=\int\!f^2$. The fact that $c'<\infty$ follows from the assumption on $r$
and $\varphi$ and the bounds
  \begin{equation}
    |\!\Ai(x)|\leq c\,e^{-\frac23x^{3/2}}\quad\text{for $x\geq 0$}\qqand
    |\!\Ai(x)|\leq c|x|^{-1/4}\quad\text{for $x<0$}\label{eq:airyBd}
  \end{equation}
  (see (10.4.59-60) in \cite{abrSteg}). This shows that for $i<j$,
  $\oQ_{t_i}e^{-(t_j-t_i)H}$ is a bounded operator mapping $L^2(\rr)$ to itself. Similar
  computations allow to check the rest of Assumption \ref{assum:0}(i). To check (ii) we
  use the formula given in Lemma \ref{lem:alt}. Each term can be written as a
  product of the form
  $(-1)^k(e^{(t_1-t_{a_0})H}\oQ_{t_{a_0}})\dotsm(e^{(t_{a_{k-1}}-t_{a_k})H}\oQ_{t_{a_k}})
  (e^{(t_{a_k}-t_{a_1})H}\KAirytwo)$ with $1\leq a_0<\dotsm<a_k\leq n$. The $k+1$ factors
  coming after $(-1)^k$ can be checked to be bounded operators on $L^2(\rr)$ by a
  computation similar to \eqref{eq:oQetH}, and a simpler computation gives the same for
  $e^{(t_{a_k}-t_{a_1})H}\KAirytwo$ using the spectral formula for its kernel.

  As in the previous example, checking Assumption \ref{assum:1} is straightforward. For Assumption \ref{assum:2} we choose
  \[V_{t_i}f(x)=U'_{t_i}f(x)=\psi(x)f(x),\qquad
  V_{t_i}'f(x)=U_{t_i}f(x)=\psi(x)^{-1}f(x)\] with
  $\psi(x)=e^{-rx/2}\uno{x\geq0}+\varphi(x)^{1/2}\uno{x<0}$. Condition (i) is obvious.  Now
  note that $\KAirytwo=B_0P_0B_0$, where $B_0(x,\lambda)=\Ai(x+\lambda)$ and $P_af(x)=f(x)\uno{x\geq
    a}$, so
  \begin{equation}
    \|V_{t_i}\oQ_{t_i}\KAirytwo V_{t_i}'\|_1\leq\|V_{t_i}\oQ_{t_i}B_0P_0\|_2\|P_0B_0V_{t_i}'\|_2.\label{eq:VtoQ}
  \end{equation}
  We have
  \begin{align*}
    \|V_{t_i}\oQ_{t_i}B_0P_0\|_2^2&\leq c\int_{-\infty}^\infty dx\int_0^\infty
    d\lambda\,\psi(x)^2\phi(x)^{2}\Ai(x+\lambda)^2\\
    &=c\int_0^\infty dx\int_0^\infty
    d\lambda\,e^{rx}\Ai(x+\lambda)^2+c\int_{-\infty}^0 dx\,\varphi(x)^{-1}\int_{-x}^\infty
    d\lambda\Ai(\lambda)^2.
  \end{align*}
  The first integral on the right-hand side is clearly finite by \eqref{eq:airyBd}. For
  the second one, note that by \eqref{eq:airyBd} $\int_{-x}^\infty
  d\lambda\Ai(\lambda)^2\leq c(1+|x|^{1/2})$, so the integral is finite by the assumption on
  $\varphi$, which by the Cauchy-Schwartz inequality implies
  $\int_{-\infty}^0dx\,|x|^{1/2}\varphi(x)^{-1}<\infty$. $\|P_0B_0V_{t_i}'\|_2<\infty$
  follows from the exact same calculation, and hence from \eqref{eq:VtoQ} we get that
  $V_{t_i}\oQ_{t_i}\KAirytwo V_{t_i}'$ is trace class. The same proof shows that
  $V_{t_i}\oQ_{t_i}e^{(t_j-t_i)H}\KAirytwo V_{t_i}'$ is trace class for $i<j$. To check
  that $V_{t_i}\oQ_{t_i}e^{-tH}V_{t_j}'$ is trace class for $i<j$ and $t=t_j-t_i>0$ we
  start by writing
  \[\|V_{t_i}\oQ_{t_i}e^{-tH}V_{t_j}'\|_2\leq\|V_{t_i}\oQ_{t_i}e^{-tH/2}\|_2\|e^{-tH/2}V_{t_j}'\|_2.\]
  For the first factor we use again the explicit formula for the kernel of $e^{-tH}$ to
  obtain
  \begin{multline}\label{eq:tcA2}
      \|V_{t_i}\oQ_{t_i}e^{-tH/2}\|_2^2=\int_{\rr^4}dx\,dy\,d\lambda\,d\sigma\,\psi(x)^2\phi(x)^2
      e^{t(\lambda+\sigma)/2}\!
      \Ai(x+\lambda)\!\Ai(y+\lambda)\!\Ai(x+\sigma)\!\Ai(y+\sigma)\\
      =\int_{-\infty}^\infty dx\int_{-\infty}^\infty d\lambda\,\psi(x)^2\phi(x)^2
        e^{t\lambda}\!\Ai(x+\lambda)^2 =\int_{-\infty}^\infty
      dx\,\psi(x)^2\phi(x)^2e^{-tx}\int_{-\infty}^\infty d\lambda\,e^{t\lambda}\!\Ai(\lambda)^2
  \end{multline}
  which is finite by the similar arguments as above. $\|e^{-tH/2}V_{t_i}'\|_2$ can
  be bounded in the same manner, and we deduce that $V_{t_i}\oQ_{t_i}e^{-tH}V_{t_j}'$ is
  trace class. The same proof works for  $V_{t_i}\oQ_{t_i}e^{-tH}\KAirytwo V_{t_j}'$.

  To get (iii) we use Lemma \ref{lem:alt} and rewrite each term in the sum as
  \begin{multline*}
    (-1)^k(U_{t_i}e^{(t_i-t_{a_0})H}\oQ_{t_{a_0}}U_{t_{a_0}}')(U_{t_{a_0}}e^{(t_{a_0}-t_{a_1})H}\oQ_{t_{a_1}}U_{t_{a_1}}')\dotsm
    (U_{t_{a_{k-1}}}'e^{(t_{a_{k-1}}-t_{a_k})H}\oQ_{t_{a_{k}}}U_{t_{a_k}}')\\\cdot(U_{t_{a_k}}e^{(t_{a_k}-t_{a_1})H}\KAirytwo U_{t_{a_1}}').
  \end{multline*}
  Since $U_{t_i}=V_{t_i}'$ and $U_{t_i}'=V_{t_i}$, each factor above corresponds to the
  adjoint of one of the factors appearing in (ii). Since the adjoint of a trace class
  operator is also trace class, we deduce that the whole product is trace class.
\end{proof}

\begin{proof}[Proof of Corollary \ref{cor:2to1}]
  We already indicated how to check Assumption \ref{assum:1}. One checks directly that the
  first three operators in Assumption \ref{assum:0}(i) are bounded operators preserving
  $L^2(\rr)$, while the last one can be checked using \eqref{eqtildeKCompleteInfinity} and
  arguing about the Airy functions appearing there similarly as in the previous proof.
  Assumption \ref{assum:0}(ii) follows similarly using Lemma \ref{lem:alt}. Assumption
  \ref{assum:2} can be checked following the same ideas as in the proof of Corollary
  \ref{Airytwobddform} and using the arguments in Appendix A of \cite{bfp} to provide the
  necessary analytical estimates.
\end{proof}

\end{document}